\documentclass{birkjour}

\usepackage[utf8]{inputenc}
\usepackage[T1]{fontenc}
\usepackage[english]{babel}
\usepackage{mathtools,amsmath,amssymb,amsfonts,amsthm,mathrsfs}
\mathtoolsset{showonlyrefs}
\usepackage{enumitem}

\usepackage{geometry}
\newtheorem{theorem}{Theorem}[section]
\newtheorem{lemma}{Lemma}[section]

\newtheorem{corollary}{Corollary}[section]
\theoremstyle{definition}
\newtheorem{example}{Example}[section]
\theoremstyle{definition}
\newtheorem{remark}{Remark}[section]
\newtheorem{definition}{Definition}[section]

\numberwithin{equation}{section}

\begin{document}

\title[Bound sets for a class of $\phi$-Laplacian operators]{Bound sets for a class of $\phi$-Laplacian operators}

\author[G.~Feltrin]{Guglielmo Feltrin}

\address{
Department of Mathematics, Computer Science and Physics, University of Udine\\
Via delle Scienze 206, 33100 Udine, Italy}

\email{guglielmo.feltrin@uniud.it}

\author[F.~Zanolin]{Fabio Zanolin}

\address{
Department of Mathematics, Computer Science and Physics, University of Udine\\
Via delle Scienze 206, 33100 Udine, Italy}

\email{fabio.zanolin@uniud.it}

\thanks{Work performed under the auspices of the Grup\-po Na\-zio\-na\-le per l'Anali\-si Ma\-te\-ma\-ti\-ca, la Pro\-ba\-bi\-li\-t\`{a} e le lo\-ro Appli\-ca\-zio\-ni (GNAMPA) of the Isti\-tu\-to Na\-zio\-na\-le di Al\-ta Ma\-te\-ma\-ti\-ca (INdAM) and supported by the project PRID \textit{SiDiA -- Sistemi Dinamici e Applicazioni} of the DMIF, University of Udine. The first author is supported by INdAM--GNAMPA project ``Problemi ai limiti per l'equazione della curvatura media prescritta''. This work was preliminarily announced on the occasion of the workshop ``Recent Advances on Dynamical Equations'' in honor of Professor Luisa Malaguti, held in Ancona, October 24--25, 2019.
\\
\textbf{Preprint -- December 2020}}

\subjclass{34B15, 34C25, 47H11.}

\keywords{Periodic solutions, continuation theorems, $\phi$-Laplacian operators, bound sets, Nagumo--Hartman condition, Li\'{e}nard and Rayleigh systems.}

\date{}

\dedicatory{}

\begin{abstract}
We provide an extension of the Hartman--Knobloch theorem for periodic solutions of vector differential systems to a general class of $\phi$-Laplacian differential operators.
Our main tool is a variant of the Man\'{a}sevich--Mawhin continuation theorem developed for this class of operator equations, together with the theory of bound sets. Our results concern the case of convex bound sets for which we show some new connections using a characterisation of sublevel sets due to Krantz and Parks.
We also extend to the $\phi$-Laplacian vector case a classical theorem of Reissig for scalar periodically perturbed Li\'{e}nard equations.
\end{abstract}

\maketitle

\section{Introduction}\label{section-1}

In the present paper, we study the existence of periodic solutions to differential systems involving a $\phi$-Laplacian differential operator, with the aim of extending the classical theorem of Hartman--Knobloch for the differential system
\begin{equation}\label{eq-x''-intro}
x'' = f(t,x,x').
\end{equation}
In his original work \cite{Ha-60} (see also \cite[ch.~XII]{Ha-book-64}), Hartman considered the two-point boundary value problem (BVP) associated with \eqref{eq-x''-intro}, assuming a (at most quadratic) growth condition of $\|f(t,x,y)\|$ in $y$ (the so-called ``Bernstein--Nagumo--Hartman condition'', cf.~\cite{Fa-85,Ma-81}) and the hypothesis that the vector field is ``repulsive'' according to
\begin{equation}\label{hy-rep}
\begin{aligned}
\langle f(t,x,y), x \rangle + \|y\|^{2} \geq 0, \quad &\text{for every $x\in\mathbb{R}^{n}$ with $\|x\|=R$} 
\\&\text{and $y\in\mathbb{R}^{n}$ such that $\langle x,y\rangle =0$.}
\end{aligned}
\end{equation}
In the special case of a Newtonian force without friction $f=f(t,x)$, condition \eqref{hy-rep} reads as
\begin{equation}\label{hy-rep-0}
\langle f(t,x), x \rangle \geq 0, \quad \text{for every $x\in\mathbb{R}^{n}$ with $\|x\|=R$.}
\end{equation}
Under the above assumptions, one find that the BVP has at least one solution in the closed ball $B[0,R]=\{x\in\mathbb{R}^{n}\colon \|x\|\leq R\}$ with center $0$ and radius $R$ (or in the open one $B(0,R)$, if the inequalities in \eqref{hy-rep} or \eqref{hy-rep-0} are strict). The result was then obtained in the context of the periodic BVP by Knobloch in \cite{Kn-71}, under the same geometric conditions on the vector field.

In the Seventies, from these results a line of research initiated and received a great interest, as shown by a series of works \cite{Be-74, BeSc-73, Ma-74, Sc-72}. In particular, assumptions \eqref{hy-rep} and \eqref{hy-rep-0}, referring to the boundary of a ball, were extended to a broader families of domains.
With this respect, very general conditions were proposed by Mawhin \cite{Ma-74} and Bebernes \cite{Be-74}, by introducing the concept of ``bounding function'', namely a Lyapunov-like function $V\colon\mathbb{R}^{n}\to\mathbb{R}$ of class $\mathcal{C}^{2}$ which plays the same role as $\|x\|^{2}/2$, so that \eqref{hy-rep} can be written as
\begin{equation}\label{hy-rep-V}
\begin{aligned}
\langle f(t,x,y), V'(x) \rangle + \langle V''(x) y, y \rangle \geq 0, 
\quad &\text{for every $x\in\mathbb{R}^{n}$ with $V(x)=c$}
\\&\text{and $y\in\mathbb{R}^{n}$ such that $\langle V'(x),y\rangle =0$,}
\end{aligned}
\end{equation}
where we have denoted with $V'$ the gradient of $V$ and by $V''$ its Hessian matrix.
In this case, the open ball $B(0,R)$ is replaced by the sublevel set $[V< c] := \{x\in\mathbb{R}^{n}\colon V(x)< c\}$.

Assumptions of the form \eqref{hy-rep} or \eqref{hy-rep-V} (typically with the strict inequality) imply that there are no solutions ``tangent from the interior'' to the boundary of a given open bounded set $G\subseteq\mathbb{R}^{n}$ (i.e., $G$ an open ball or a sublevel set in the above examples), namely the vector field satisfies Wa\.{z}ewski-type conditions at the boundary (see \cite{KaLaYo-74}, as well as \cite{Sr-04} and the references therein). 
This in turn lead to the concept of \textit{bound set} introduced in 
\cite{GaMa-77}, which has found many relevant applications to first and second order differential systems.
Extensions of these methods have been obtained by several authors in various directions such as Floquet BVPs, Carath\'{e}odory conditions, non-smooth bounding functions and non-smooth boundaries, differential inclusions, extensions to Banach spaces, PDEs, integro-differential equations, differential equations with impulses, nonlocal BVPs. A constantly growing literature in this area shows the persistent influence of the bound sets and bounding functions approach; see \cite{AnKoMa-09,AnMaPa-09,BLMT-17,BMT-19,MaSD-17,MaSD-19,MaSD-19rimut,PaTa-19} and the references therein.

The aim of our work is to propose further applications of the bound set technique to the study of the periodic BVP associated with the second-order $\phi$-Laplacian differential operator
\begin{equation}\label{eq-phi-intro}
( \phi(x') )' = f(t,x,x'),
\end{equation}
where $\phi \colon \mathbb{R}^{n} \to \phi(\mathbb{R}^{n})=\mathbb{R}^{n}$
is a homeomorphism such that $\phi(0)=0$, and $f \colon \mathbb{R}\times \mathbb{R}^{n}\times \mathbb{R}^{n}\to \mathbb{R}^{n}$
is a continuous vector field which is $T$-periodic in the $t$-variable. 
The study of nonlinear scalar differential equations with a $p$-Laplacian operator, or more generally of the form
\begin{equation*}
\mathrm{div} \mathcal{A}(\xi ,u,u_{\xi}) = \mathcal{B}(\xi,u,u_{\xi}),
\end{equation*}
for $\xi$ in a domain of $\mathbb{R}^{N}$, is a classical topic in the area of PDEs, see~\cite{Se-64} (cf.~also \cite{PuSe-07} and the references therein).
In the setting of ODEs differential systems, a systematic investigation of the periodic problem was initiated by Man\'{a}sevich and Mawhin in \cite{MaMa-98}. It is noteworthy that the analysis of systems has an independent interest also from the point of view of the analysis of differential equations in the complex plane or in the complex space $\mathbb{C}^{n}$ (cf.~\cite{Ma-12}).

Our contribution pursues a line of research started by Mawhin \cite{Ma-00} and by Mawhin and Ure\~{n}a \cite{MaUr-02}, dealing with the $p$-Laplacian operator $\phi(s)=\phi_{p}(s):=|s|^{p-2}s$ with $p>1$, with the aim of further extending Hartman--Knobloch theorem to a broader class of nonlinear differential operators.
A partial result in this direction was obtained in \cite[Theorem~5.2]{FeZa-17} when the domain is a ball. Here, we extend also this result to more general domains.

\medskip

The plan of the paper is as follows. In Section~\ref{section-2}, we introduce the main theoretical tools for the solvability of the periodic BVP associated with \eqref{eq-phi-intro}, which are based on the Man\'{a}sevich--Mawhin continuation theorems \cite{MaMa-98} in the versions stated in \cite{FeZa-17} for a general homeomorphism $\phi$.
This leads to Theorem~\ref{th-bs} and Theorem~\ref{th-bs-bis}, after the introduction of the concept of bound sets and Nagumo--Hartman condition for equation \eqref{eq-phi-intro}. As a next step, following \cite{GaMa-77,Ma-74,Ma-stud-77,Ma-79}, we introduce a family of bounding functions and obtain the corresponding Theorem~\ref{th-bs2} and Theorem~\ref{th-bs3}.
Section~\ref{section-3} is devoted to the application of these latter results. In particular, for a set $G:=[V<0]$ and the equation
\begin{equation*}
( \phi(x') )' = f(t,x),
\end{equation*}
we find the existence of a $T$-periodic solution with values in $\overline{G}$, provided that 
\begin{itemize}[leftmargin=28pt,labelsep=10pt]
\item[$(i)$] $V'(x)\neq 0$, for every $x\in[V=0]$;
\item[$(ii)$] $\langle V''(x)y,y\rangle \geq 0$, for every $x\in \partial G$ and $y\in \mathbb{R}^{n}$ with $\langle V'(x),y\rangle = 0$;
\item[$(iii)$] $\langle V'(x),f(t,x)\rangle \geq 0$, for every $t\in \mathopen{[}0,T\mathclose{]}$ and $x\in \partial G$.
\end{itemize}
The class of $\phi$-Laplacians for which the above result holds require that $\phi(\xi)=A(\xi)\xi$, $\xi\in\mathbb{R}^{n}$, with $A$ a positive scalar function. As shown in \cite{FeZa-17}, maps of this form are not necessarily monotone and include all the $p$-Laplacian vector differential operators, thus extending Hartman--Knobloch theorem for frictionless vector fields along the line of \cite{Ma-00}.
The case of a vector field depending on $x'$ is studied in Section~\ref{section-3.2} extending \cite[Corollary~6.3]{Ma-74} for vector Rayleigh systems to the $\phi$-Laplacians. Another result for vector $\phi$-Laplacian Li\'{e}nard systems
is given in Section~\ref{section-3.3} extending \cite[Theorem~3.3]{PeXu-07} as well as a classical theorem of Reissig \cite[Theorem~3]{Re-75}.

We observe that conditions $(i)$ and $(ii)$ appear explicitly or implicitly in some classical results for equation \eqref{eq-x''-intro}, see for instance \cite{Be-74,HaSc-83,Ma-74}. Indeed, consider condition \eqref{hy-rep-V} already assumed in the above quoted papers in the special case when $f=f(t,x)$. From 
\begin{equation}\label{eq-M-intro}
\langle V''(x)y,y\rangle \geq - \langle V'(x),f(t,x)\rangle\geq -M, 
\end{equation}
(with $M:=\sup\{\langle V'(x),f(t,x)\rangle \colon t\in\mathopen{[}0,T\mathclose{]}, x\in \overline{G}\}$), we obtain that the quadratic form $y\mapsto \langle V''(x)y,y\rangle$ is bounded from below and hence positive semi-definite, so that $(ii)$ holds. Notice also that $(iii)$ follows reading \eqref{hy-rep-V} for $y=0$.
Conditions $(i)$ and $(ii)$ hint some kind of convexity properties for the sublevel sets of $V$, however, at the best of our knowledge, there are no explicit proofs relating these conditions with the convexity of the set $[V\leq0]$.
In general, the convexity of the sublevel set is not guaranteed by the sole conditions $(i)$ and $(ii)$, as it could be shown by an example of a sublevel set made by the union of two disjoint balls. On the other hand, if we assume that $[V\leq0]$ is connected, then the convexity is achieved thanks to an analysis due to Krantz and Parks in \cite{KrPa-99} (see also \cite{Kr-1011}). For the reader's convenience, we propose a different proof of this result in Appendix~\ref{appendix-A}.
We exploit this characterization of the convexity (see Lemma~\ref{lem-convex}) in the proof of Corollary~\ref{cor-V}.
Clearly, for a general $f$ depending on $y=x'$, or if we have some information on the a priori bounds on $x'$ (say $\|x'\|_{\infty}\leq K$), we are not allowed to conclude with the positive semi-definiteness of $V''(x)$ from \eqref{eq-M-intro} because only the $y$ with $\|y\|\leq K$ will be involved in \eqref{eq-M-intro}. An investigation of non-convex bound sets was proposed in \cite{AmHa-13} (see also \cite{Za-87ud} for some remarks about this problem).

Throughout the paper, we denote by $\|\cdot\|$ the Euclidean norm in $\mathbb{R}^{n}$ and by $\langle\cdot,\cdot\rangle$ the associated standard inner product.

\section{Main results}\label{section-2}

We consider the vector differential equation
\begin{equation}\label{eq-phi}
( \phi(x') )' = f(t,x,x'),
\end{equation}
where
$\phi \colon \mathbb{R}^{n} \to \phi(\mathbb{R}^{n})=\mathbb{R}^{n}$
is a homeomorphism such that $\phi(0)=0$, and $f \colon \mathbb{R}\times \mathbb{R}^{n}\times \mathbb{R}^{n}\to \mathbb{R}^{n}$
is a continuous vector field which is $T$-periodic in the $t$-variable.

We study the problem of the existence of $T$-periodic solutions for
\eqref{eq-phi}, which, equivalently, can be reduced to the search of
solutions of \eqref{eq-phi} satisfying the boundary condition
\begin{equation}\label{eq-per}
x(0) = x(T), \quad x'(0) = x'(T).
\end{equation}
By a solution of \eqref{eq-phi} we mean a continuously differentiable function $x(\cdot)$ with $\phi(x'(\cdot))$ continuously differentiable and
satisfying \eqref{eq-phi} for all $t$.
With this respect, it is useful to introduce the space
$\mathcal{C}^{1}_{T}$ of the continuously differentiable functions $x \colon \mathopen{[}0,T\mathclose{]}\to \mathbb{R}^{n}$
satisfying the boundary condition \eqref{eq-per}, endowed with the $\mathcal{C}^{1}$-norm
\begin{equation*}
\|x\|_{\mathcal{C}^{1}}:= \|x\|_{\infty} + \|x'\|_{\infty},
\end{equation*}
where, for a continuous function $u \colon \mathopen{[}0,T\mathclose{]}\to \mathbb{R}^{n}$, we denote
\begin{equation*}
 \|u\|_{\infty} := \sup_{t\in\mathopen{[}0,T\mathclose{]}} \|u(t)\|
\end{equation*}
the classical $\sup$-norm.

Equivalently, instead of \eqref{eq-phi}, we can consider the
differential system in $\mathbb{R}^{n}\times \mathbb{R}^{n}$
\begin{equation}\label{syst-eq}
\begin{cases}
\, x' = \phi^{-1}(y) \\
\, y' = f(t,x,\phi^{-1}(y))
\end{cases}
\end{equation}
and look for a (continuously differentiable) solution $(x(\cdot),y(\cdot))$.
Every $T$-periodic solution $x(\cdot)$ of \eqref{eq-phi}
corresponds to the $T$-periodic solution $(x(\cdot),y(\cdot))$ of
\eqref{syst-eq} with $y(\cdot)= \phi(x'(\cdot))$. Note that a
solution of \eqref{eq-phi} defined on $\mathopen{[}0,T\mathclose{]}$
is bounded in the $\mathcal{C}^{1}$-norm if and only if
the corresponding solution of \eqref{syst-eq} is bounded in the uniform norm
of $\mathcal{C}(\mathopen{[}0,T\mathclose{]},\mathbb{R}^{2n})$.

For the existence of $T$-periodic solutions of \eqref{eq-phi}, we apply the
Man\'{a}sevich--Mawhin continuation theorems \cite{MaMa-98} in the versions
elaborated in \cite{FeZa-17} where weaker conditions on $\phi$ are assumed.

\begin{theorem}\label{th-cont}
Let $F = F(t,x,y;\lambda) \colon \mathopen{[}0,T\mathclose{]} \times \mathbb{R}^{n}\times \mathbb{R}^{n}\times \mathopen{[}0,1\mathclose{]}\to \mathbb{R}^{n}$ be a continuous function such that
\begin{equation*}
F(t,x,y;1) = f(t,x,y), \qquad F(t,x,y;0) = f_{0}(x,y),
\end{equation*}
where $f_{0} \colon \mathbb{R}^{n}\times \mathbb{R}^{n} \to \mathbb{R}^{n}$
is an autonomous vector field.
Let $\Omega \subseteq \mathcal{C}^{1}_{T}$ be an open and bounded set. Suppose that
\begin{itemize}[leftmargin=28pt,labelsep=10pt]
\item[$(\textsc{h}_{1})$] for each $\lambda\in \mathopen{[}0,1\mathclose{[}$ the problem
\begin{equation*}
\begin{cases}
\, (\phi(x'))' = F(t,x,x';\lambda), \\
\, x(0) = x(T), \quad x'(0) = x'(T),
\end{cases}
\eqno{(P_{\lambda})}
\end{equation*}
has no solution $x\in \partial\Omega$;
\item[$(\textsc{h}_{2})$] the condition on the Brouwer degree
\begin{equation*}
\mathrm{d}_{\mathrm{B}}(f_{0}(\cdot,0), \Omega\cap \mathbb{R}^{n},0)\neq0
\end{equation*}
holds.
\end{itemize}
Then, problem \eqref{eq-phi}-\eqref{eq-per}
has at least a solution in $\overline{\Omega}$.
\end{theorem}

Theorem~\ref{th-cont} corresponds to \cite[Theorem~4.1]{MaMa-98} (see also \cite[Theorem~4.6]{FeZa-17}). We recall also another version which is strictly related to the classical Mawhin continuation theorem \cite{Ma-69,Ma-93}, in an appropriate form for the $\phi$-Laplacian operators given in \cite[Theorem~3.1]{MaMa-98} (see also \cite[Theorem~3.11]{FeZa-17}).

\begin{theorem}\label{th-cont-2}
Let $F = F(t,x,y) \colon \mathopen{[}0,T\mathclose{]} \times \mathbb{R}^{n}\times \mathbb{R}^{n}\to \mathbb{R}^{n}$ be a continuous function. 
Let $\Omega \subseteq \mathcal{C}^{1}_{T}$ be an open and bounded set. Suppose that
\begin{itemize}[leftmargin=28pt,labelsep=10pt]
\item[$(\textsc{h}'_{1})$] for each $\lambda\in \mathopen{]}0,1\mathclose{[}$ the problem
\begin{equation*}
\begin{cases}
\, (\phi(x'))' = \lambda F(t,x,x'), \\
\, x(0) = x(T), \quad x'(0) = x'(T),
\end{cases}
\eqno{(P'_{\lambda})}
\end{equation*}
has no solution $x\in \partial\Omega$;
\item[$(\textsc{h}'_{2})$] the condition on the Brouwer degree
\begin{equation*}
\mathrm{d}_{\mathrm{B}}(F^{\#}, \Omega\cap \mathbb{R}^{n},0)\neq0
\end{equation*}
holds, where $F^{\#}(s):=\frac{1}{T}\int_{0}^{T}F(t,s,0)\,\mathrm{d}t$.
\end{itemize}
Then, problem \eqref{eq-phi}-\eqref{eq-per}
has at least a solution in $\overline{\Omega}$.
\end{theorem}

We propose now some applications of the above continuation theorems which are motivated by
the theory of bound sets (cf.~\cite{GaMa-77,Ma-stud-77,Ma-79}). We shall focus our attention to the case of Theorem~\ref{th-cont}. 
Similar applications can be given starting from Theorem~\ref{th-cont-2}.

The bound set approach represents a general method to verify the abstract condition of the non-existence of solutions in $\partial\Omega$, given in $(\textsc{h}_{1})$ (respectively $(\textsc{h}'_{1})$), by introducing a more concrete condition of the non-existence of solutions tangent to the boundary of an open and bounded set $G\subseteq\mathbb{R}^{n}$. In this setting, we also introduce some Nagumo--Hartman conditions (cf.~\cite{Ha-60,MaUr-02}), which are classical in this framework (cf.~\cite{Ma-81}) and allow to find a priori bounds for $\|x'\|_{\infty}$.

Let $G\subseteq \mathbb{R}^{n}$ be an open and bounded set.
Following \cite{Ma-74} (see also \cite{Za-87ud}), we say that the
system
\begin{equation*}
(\phi(x'))' = F(t,x,x';\lambda) \eqno{(E_{\lambda})}
\end{equation*}
is a \textit{Nagumo equation with respect to $G$} (with constant $K$) if there exists a constant $K>0$ such that, for every
$\lambda\in \mathopen{[}0,1\mathclose{[}$ and for every solution $x(\cdot)$ of problem $(P_{\lambda})$, with $x(t)\in \overline{G}$
for all $t\in \mathopen{[}0,T\mathclose{]}$, it holds that $\|x'\|_{\infty} \leq K$.

If the Nagumo condition is satisfied, from Theorem~\ref{th-cont}
we deduce the next result.

\begin{theorem}\label{th-bs}
Let $F = F(t,x,y;\lambda) \colon \mathopen{[}0,T\mathclose{]} \times \mathbb{R}^{n}\times \mathbb{R}^{n}\times \mathopen{[}0,1\mathclose{]}\to \mathbb{R}^{n}$ be a continuous function such that
\begin{equation*}
F(t,x,y;1) = f(t,x,y), \qquad F(t,x,y;0) = f_{0}(x,y),
\end{equation*}
where $f_{0} \colon \mathbb{R}^{n}\times \mathbb{R}^{n} \to \mathbb{R}^{n}$
is an autonomous vector field. Suppose that there exists an open bounded set
$G\subseteq \mathbb{R}^{n}$ such that
\begin{itemize}[leftmargin=34pt,labelsep=10pt]
\item[$(\textsc{h}_{\mathrm{N}})$] the system $(E_{\lambda})$ is a Nagumo equation with respect to $G$;
\item[$(\textsc{h}_{\mathrm{BS}})$] for each $\lambda\in \mathopen{[}0,1\mathclose{[}$ there is no solution of $(P_{\lambda})$ such
that $x(t)\in\overline{G}$ for all $t\in\mathopen{[}0,T\mathclose{]}$
and $x(t_{0})\in\partial G$ for some
$t_{0}\in\mathopen{[}0,T\mathclose{]}$;
\item[$(\textsc{h}_{\mathrm{D}})$]  $\mathrm{d}_{\mathrm{B}}(f_{0}(\cdot,0),G,0)\neq0$.
\end{itemize}
Then, problem \eqref{eq-phi}-\eqref{eq-per}
has at least a solution $\tilde{x}$ such that $\tilde{x}(t)\in \overline{G}$,
for all $t\in\mathopen{[}0,T\mathclose{]}$.
\end{theorem}

In the sequel, we  refer to $(\textsc{h}_{\mathrm{BS}})$ as a \textit{bound set condition}.
As previously observed, this condition of non-tangency of the solutions at the boundary of $G$ replaces condition $(\textsc{h}_{1})$ of Theorem~\ref{th-cont}.

\begin{proof}
Our argument borrows the classical scheme in \cite{Ma-74}. We are going to apply Theorem~\ref{th-cont}.
According to the Nagumo condition $(\textsc{h}_{\mathrm{N}})$, there exists a constant $K>0$ such that all the solutions of $(P_{\lambda})$ with values in $\overline{G}$ have the derivative bounded by $K$. Hence, we define the set of functions
\begin{equation*}
\Omega:= \bigl{\{}x\in \mathcal{C}^{1}_{T} \colon \text{$x(t)\in G$ for all $t\in \mathopen{[}0,T\mathclose{]}$, $\|x'\|_{\infty} < K+1$} \bigr{\}},
\end{equation*}
which is open and bounded in $\mathcal{C}^{1}_{T}$ (cf.~\cite{GaMa-77, Ma-79}).
In order to check that condition $(\textsc{h}_{1})$ holds, observe that
if (by contradiction) a solution $x$ of $(P_{\lambda})$ satisfies $x\in \partial\Omega$, then $x\in \overline{\Omega}$, hence $x(t)\in \overline{G}$ for all $t$.
Then, by the Nagumo condition, $\|x'\|_{\infty} \leq K < K+1$. By the
definition of $\Omega$ and the assumption $x\in \partial\Omega$, we cannot
have $x(t)\in G$ for all $t$ and therefore there exists
$t_{0}\in \mathopen{[}0,T\mathclose{]}$ such that $x(t_{0})\in \partial G
= \overline{G}\setminus G$. This situation is not possible in view of the
bound set condition $(\textsc{h}_{\mathrm{BS}})$.

To check $(\textsc{h}_{2})$ we simply observe that $\Omega\cap \mathbb{R}^{n} = G$
and thus apply $(\textsc{h}_{\mathrm{D}})$. This concludes the proof.
\end{proof}

The concept of \textit{Nagumo equation} was introduced by Mawhin in \cite{Ma-74} and further developed in \cite{Ma-81} as a
generalization of the classical \textit{Nagumo--Hartman condition}
\cite{Ha-60,Ha-book-64}. This latter condition was originally expressed as a growth restriction on the vector field $f(t,x,y)$
in order to provide an a priori bound on $\|x'(t)\|$ for the solutions of the second-order vector differential equation
\begin{equation}\label{eq-linear}
x'' = f(t,x,x')
\end{equation}
with $\|x(t)\|$ uniformly bounded.

For the periodic boundary value problem associated with \eqref{eq-linear},
the Nagumo--Hartman condition reads as follows: \textit{given $R>0$,
\begin{itemize}[leftmargin=28pt,labelsep=10pt]
\item[$(\textsc{nh}_1)$]
there exists a continuous function $\eta \colon \mathopen{[}0,+\infty\mathclose{[}\to \mathopen{]}0,+\infty\mathclose{[}$
such that
\begin{equation*}
\int^{+\infty} \dfrac{s}{\eta(s)}\,\mathrm{d}s = +\infty \; \text{ and }\;
\|f(t,x,y)\|\leq \eta(\|y\|), \;\forall\, t\in \mathopen{[}0,T\mathclose{]}, \;\|x\|\leq R, \;
y\in \mathbb{R}^{n};
\end{equation*}
\item[$(\textsc{nh}_2)$]
if $n>1$, there are non-negative constants
$\alpha, \beta$ such that
\begin{equation*}
\|f(t,x,y)\|\leq 2\alpha\bigl{(}\langle x,f(t,x,y)\rangle + \|y\|^2 \bigr{)} + \beta,
\quad\forall\, t\in \mathopen{[}0,T\mathclose{]}, \;\|x\|\leq R, \;
y\in \mathbb{R}^{n}.
\end{equation*}
\end{itemize}}
Under these assumptions  it holds that for every $R > 0$ there exists a constant $K > 0$ (depending on $R, \eta, \alpha, \beta, T$) such that every $T$-periodic solution $x(\cdot)$ of \eqref{eq-linear} satisfying $\|x\|_{\infty}\leq R$
is such that $\|x'\|_{\infty}\leq K$ (cf.~\cite{Be-74,Kn-71,Ma-74}). Observe that condition $(\textsc{nh}_2)$
implies that the solution $x(\cdot)$ satisfies the constraint
$\|x''(t)\| \leq \alpha \tfrac{d^2}{dt^2}\|x(t)\|^2 + \beta$ for all $t$.

In the setting of Theorem~\ref{th-bs}, for the special case $\phi(\xi)=\xi$, the condition should be applied with $F(t,x,y;\lambda)$ in place of $f(t,x,y)$ (and uniformly with respect to $\lambda\in\mathopen{[}0,1\mathclose{]}$). 

As observed in \cite[Proposition~5.2]{Ma-74},
the concept of Nagumo equation is more general as it covers some second-order differential systems for which $(\textsc{nh}_1)$-$(\textsc{nh}_2)$ are not satisfied.
Extensions of the Nagumo--Hartman conditions to more general differential operators have been obtained in more recent years, see, for
instance, \cite{MaUr-02} dealing with the vector $p$-Laplacian, and \cite{MaTh-11} for more general scalar nonlinear differential operators.
In previous works, it has been provided precise growth assumptions for the vector field $f$, generalizing to the $p$-Laplacian type operator the classical conditions $(\textsc{nh}_1)$-$(\textsc{nh}_2)$, for the second order linear differential operator (cf.~\cite[conditions $(b)$-$(c)$]{MaUr-02}).
Typically these modified Hartman--Nagumo conditions involve a growth assumption on $f$ with respect to the $y$-variable which is related to the exponent $p$ in the $p$-Laplacian operator. For a general $\phi$-Laplacian the situation appears more complicated. Indeed, the following example shows that, for any homeomorphism $\phi$ of the real line having a power-growth at infinity, we can determine a suitable growth-rate in $x'$ such that the Hartman--Nagumo condition is not satisfied.
 
\begin{example}\label{example-1}
Let $\phi \colon \mathbb{R} \to \phi(\mathbb{R})=\mathbb{R}$ be an increasing homeomorphism such that $\phi(0)=0$.
Suppose also that $\int_{1}^{+\infty} \phi^{-1}( \xi)\xi^{-\frac{1+\gamma}{\gamma}} \,\mathrm{d}\xi < \infty$, for some $\gamma>0$.
Then, the differential equation
\begin{equation}\label{eq-example}
( \phi(x') )' = \gamma \bigl{(} \phi(x') \bigr{)}^{\!\frac{1+\gamma}{\gamma}}, \quad \text{in $\mathopen{]}0,1\mathclose{[}$,}
\end{equation}
has bounded solutions $x(\cdot)$, with $x'(\cdot)$ as well as $\phi(x'(\cdot))$ unbounded. Indeed, setting $u(t):=\phi(x'(t))$ we find that $u(t) = (1-t)^{-\gamma}$ solves the equation $u'= \gamma \, u^{\frac{1+\gamma}{\gamma}}$ with $u(0)=1$.
Next, we obtain $x'(t)=\phi^{-1}((1-t)^{-\gamma})$ and thus 
\begin{equation*}
x(t) = \int_{0}^{t} \phi^{-1}((1-s)^{-\gamma})\,\mathrm{d}s
\end{equation*}
is a solution of \eqref{eq-example} such that $x(0)=0$, $x'(0)=\phi^{-1}(1)$, with $x(\cdot)$ bounded in $\mathopen{[}0,1\mathclose{]}$ and $x'(t)\to+\infty$ for $t\to1^{-}$. The example can be modified in order to consider the case of periodic solutions as well.
Clearly, if $\phi(\xi)\sim \xi^{\beta}$ at $+\infty$, then taking $0<\gamma<\beta$ we provide an example where there are bounded solutions with unbounded derivatives.
\hfill$\lhd$
\end{example}

Due to the great generality of our differential operator $\phi$, we prefer to not propose a specific growth assumption, like in the above quoted papers. On the other hand, in Section~\ref{section-3}, we provide some specific application where conditions $(\textsc{nh}_1)$-$(\textsc{nh}_2)$ can be checked by a direct inspection.

A slightly variant of Theorem~\ref{th-bs} can be obtained using Theorem~\ref{th-cont-2} and referring the concept of Nagumo equation to the system
\begin{equation*}
(\phi(x'))' = \lambda F(t,x,x').
\eqno{(E'_{\lambda})}
\end{equation*}

\begin{theorem}\label{th-bs-bis}
Let $F = F(t,x,y) \colon \mathopen{[}0,T\mathclose{]} \times \mathbb{R}^{n}\times \mathbb{R}^{n}\to \mathbb{R}^{n}$ be a continuous function. Suppose that there exists an open bounded set $G\subseteq \mathbb{R}^{n}$ such that
\begin{itemize}[leftmargin=28pt,labelsep=10pt]
\item[$(\textsc{h}'_{\mathrm{N}})$] the system $(E'_{\lambda})$ is a Nagumo equation with respect to $G$;
\item[$(\textsc{h}'_{\mathrm{BS}})$] for each $\lambda\in \mathopen{]}0,1\mathclose{[}$ there is no solution of $(P'_{\lambda})$ such
that $x(t)\in\overline{G}$ for all $t\in\mathopen{[}0,T\mathclose{]}$
and $x(t_{0})\in\partial G$ for some
$t_{0}\in\mathopen{[}0,T\mathclose{]}$;
\item[$(\textsc{h}'_{\mathrm{D}})$]  $\mathrm{d}_{\mathrm{B}}(F^{\#},G,0)\neq0$.
\end{itemize}
Then, problem \eqref{eq-phi}-\eqref{eq-per}
has at least a solution $\tilde{x}$ such that $\tilde{x}(t)\in \overline{G}$,
for all $t\in\mathopen{[}0,T\mathclose{]}$.
\end{theorem}

In the sequel, the following lemma will be used as a technical step to provide the desired bounds on $\|x'\|_{\infty}$ in the context of condition $(\textsc{h}_N)$. 
We denote by $\mathbb{S}^{n-1} :=\partial B(0,1)$ the unit sphere in $\mathbb{R}^{n}$.

\begin{lemma}\label{lem-2.1}
Let $\mathcal{F}$ be a family of absolutely continuous and $T$-periodic functions such that
\begin{itemize}[leftmargin=28pt,labelsep=10pt]
\item there exists $M_{0}\geq 0$ such that for all $\omega\in\mathbb{S}^{n-1}$ and $z\in\mathcal{F}$ there exists $t_{0}\in\mathopen{[}0,T\mathclose{]}$ such that $| \langle z(t_{0}),\omega\rangle| \leq M_{0}$;
\item there exist $M_{1}>0$ and $p\in\mathopen{[}1,+\infty\mathclose{]}$ such that $\|z'\|_{L^{p}} \leq M_{1}$ for every $z\in\mathcal{F}$.
\end{itemize}
Then, there exists $K=K(M_{0},M_{1})>0$ such that $\|z\|_{\infty}\leq K$ for every $z\in\mathcal{F}$.
\end{lemma}

\begin{proof}
Let $z\in\mathcal{F}$ with $z\not\equiv 0$. Let $\hat{t}\in\mathopen{[}0,T\mathclose{]}$ be such that $\|z(\hat{t})\|=\|z\|_{\infty}\neq0$. We set $\omega=z(\hat{t})/\|z(\hat{t})\|$ and $t_{0}=t_{0}(\omega,z)$ as in the first hypothesis. Therefore, we have
\begin{align*}
\|z\|_{\infty}=\|z(\hat{t})\|=\bigl{|} \langle z(\hat{t}),\omega\rangle \bigr{|} 
&= \biggl{|} \langle z(t_{0}),\omega\rangle +\int_{t_{0}}^{\hat{t}} \langle z'(s),\omega\rangle \,\mathrm{d}s \biggr{|}
\\
&\leq M_{0} + \|z'\|_{L^{1}} \leq M_{0}+ T^{\frac{p-1}{p}} \|z'\|_{L^{p}} \leq M_{0}+T^{\frac{p-1}{p}} M_{1}.
\end{align*}
Setting $K:=M_{0}+T^{\frac{p-1}{p}}M_{1}$ (which does not depend on $z$), the conclusion holds.
\end{proof}

\medskip

We discuss now a technique introduced and developed in \cite{GaMa-77,Ma-74,Ma-stud-77,Ma-79} to verify the bound set condition.
It consists in controlling locally the solutions of $(E_{\lambda})$ at the boundary of $G$, by means of suitable Lyapunov-like functionals which are usually called \textit{bounding functions}
(see also \cite{Ma-stud-77} for an introduction to this topic).
To this end we give the following definition (see also \cite{FeZa-88,Za-87,Za-87ud}).

\begin{definition}\label{def-bf}
Let $G\subseteq \mathbb{R}^{n}$ be an open and bounded set.
Assume that for each $u\in\partial G$, there exist an open ball
$B(u,r_{u})$ of center $u$ and radius $r_{u} > 0$ and a function
$V_{u} \colon B(u,r_{u})\to \mathbb{R}$ such that $V_{u}(u) = 0$ and
\begin{equation*}
\overline{G} \cap B(u,r_{u}) \subseteq \bigl{\{} x\in B(u,r_{u}) \colon V_{u}(x)\leq 0 \bigr{\}}.
\end{equation*}
In this case, the family $(V_{u})_{u\in\partial G}$ is called
a \textit{set of bounding functions} for $G$.
\end{definition}

We are in position to present an application of the method of bounding functions
to the periodic problem associated with \eqref{eq-phi} for a homeomorphism $\phi$ having the following form
\begin{equation}\label{phi-A}
\phi(\xi) = A(\xi)\xi,\quad \text{for every $\xi\in\mathbb{R}^{n}\setminus\{0\}$,}
\qquad \phi(0)=0,
\end{equation}
where
$A \colon \mathbb{R}^{n}\setminus\{0\}\to \mathopen{]}0,+\infty\mathclose{[}$
is a continuous function.
As shown in \cite{FeZa-17} this case includes most nonlinear differential
operators considered in the literature, in particular the vector $p$-Laplacians.
Moreover, observe also that \eqref{phi-A} does not imply that the operator $\phi$ is monotone (cf.~\cite[Section~5]{FeZa-17}).

The following result holds.

\begin{lemma}\label{lem-bf1}
Let $\phi$ be
a homeomorphism of $\mathbb{R}^{n}$ of the form \eqref{phi-A}. 
Let $F = \break F(t,x,y;\lambda) \colon \mathopen{[}0,T\mathclose{]}
\times \mathbb{R}^{n}\times \mathbb{R}^{n}\times
\mathopen{[}0,1\mathclose{]}\to \mathbb{R}^{n}$
be a continuous function and let $(V_{u})_{u\in\partial G}$
be a family of bounding functions of class $\mathcal{C}^2$
for an open bounded set $G\subseteq \mathbb{R}^{n}$.
Suppose that
\begin{itemize}[leftmargin=30pt,labelsep=10pt]
\item [$(\textsc{h}_{V})$] for every $u\in \partial G$, $t\in \mathopen{[}0,T\mathclose{]}$ and $\lambda\in \mathopen{[}0,1\mathclose{[}$
\begin{equation*}
\langle V''_{u}(u)y,\phi(y)\rangle + \langle V'_{u}(u),F(t,u,y;\lambda)\rangle > 0,
\quad \forall\, y\in \mathbb{R}^{n} \colon \langle V'_{u}(u),y\rangle = 0.
\end{equation*}
\end{itemize}
Then, the bound set condition $(\textsc{h}_{\mathrm{BS}})$ holds with respect to problem $(P_\lambda)$.
\end{lemma}

\begin{proof}
By contradiction, let us suppose that $(\textsc{h}_{\mathrm{BS}})$ is not valid.
Therefore, for some $\lambda\in \mathopen{[}0,1\mathclose{[}$
there exists a solution $x(\cdot)$ of $(P_{\lambda})$
such that $x(t)\in\overline{G}$ for all
$t\in\mathopen{[}0,T\mathclose{]}$ and $x(t_{0})\in\partial G$ for some
$t_{0}\in\mathopen{[}0,T\mathclose{]}$.
Due to the boundary condition $x(0) = x(T)$, without loss of generality, we can suppose $0\leq t_{0} < T$.
We consider the point $u:=x(t_{0})\in\partial G$, the ball $B(u,r_{u})$
and the function $V_{u}$ according to Definition~\ref{def-bf}.
Thus, there exists an open neighborhood $U$ of $t_{0}$
such that $V_{u}(x(t))\leq 0$ for all $t\in U\cap \mathopen{[}0,T\mathclose{]}$ and $V_{u}(x(t_{0}))=0$,
so that $t_{0}$ is a point of maximum for the function
$v \colon U\cap \mathopen{[}0,T\mathclose{]}\to \mathbb{R}$,
$v(t):= V_{u}(x(t))$. By the chain rule, we have
$v'(t) = \langle V_{u}'(x(t)),x'(t)\rangle$ for all $t\in U\cap \mathopen{[}0,T\mathclose{]}$.
If $0<t_{0} < T$, then $v'(t_{0}) = 0$ and therefore
\begin{equation}\label{eq-2.5}
\langle V_{u}'(u),y\rangle = 0, \quad\text{for $y = x'(t_{0})$.}
\end{equation}
On the other hand, if $t_{0} = 0$ then both $0$ and $T$ are maximum points
for the differentiable function $v(\cdot)$ which now
is defined on a neighborhood of both points.
Since $v(0) = v(T) = 0$ and $v(t)\leq 0$ for $t$ in a right neighborhood of
$0$ and in a left neighborhood of $T$, then
$v'(0)=v'(0^+)\leq 0 \leq v'(T^-)=v'(T)$ and,
using the boundary condition $x'(0) = x'(T)$, we obtain again \eqref{eq-2.5}.

As a next step, we introduce the auxiliary function
\begin{equation*}
w(t):= \langle \phi(x'(t)),V_{u}'(x(t))\rangle.
\end{equation*}
Observe that $w(t_{0}) = 0$. This is trivial if $x'(t_{0}) = 0$, because $\phi(0)=0$.
If $x'(t_{0})\neq0$, then $w(t_{0}) = A(x'(t_{0}))\langle x'(t_{0}),V_{u}'(x(t_{0}))\rangle = 0$, by \eqref{eq-2.5}. 
By differentiating $w(\cdot)$ in a neighborhood of $t_{0}$ we obtain
\begin{equation*}
\begin{aligned}
w'(t) &=\dfrac{\mathrm{d}}{\mathrm{d}t} \langle \phi(x'(t)),V_{u}'(x(t))\rangle
\\
&= \langle ( \phi(x'(t)) )', V_{u}'(x(t)) \rangle
+ \langle \phi(x'(t)),V_{u}''(x(t)) x'(t)\rangle
\\
&= \langle F(t,x(t),x'(t);\lambda), V_{u}'(x(t)) \rangle
+ \langle V_{u}''(x(t)) x'(t),\phi(x'(t))\rangle.
\end{aligned}
\end{equation*}
Hence, from $(\textsc{h}_{V})$ and \eqref{eq-2.5} we find that $w'(t_{0}) > 0$. This implies that there exists an interval $\mathopen{[}t_{0},t_{0}+\varepsilon\mathclose{[}\subseteq U\cap \mathopen{[}0,T\mathclose{]}$ such that
$w(t) > 0$ for all $t\in \mathopen{]}t_{0},t_{0}+\varepsilon\mathclose{[}$.
From the definition of $w(\cdot)$ and $\phi(0)=0$, we must have $x'(t)\neq0$ for all $t \in \mathopen{]}t_{0},t_{0}+\varepsilon\mathclose{[}$ and hence
\begin{equation*}
w(t) = A(x'(t))\langle x'(t),V'_{u}(x(t))\rangle > 0, \quad \text{for all $t \in \mathopen{]}t_{0},t_{0}+\varepsilon\mathclose{[}$.}
\end{equation*}
Using the fact that $A(\xi) > 0$ for all $\xi\neq0$, we obtain that the map
\begin{equation*}
t\mapsto \langle x'(t),V'_{u}(x(t))\rangle = \dfrac{\mathrm{d}}{\mathrm{d}t}V_{u}(x(t)) = v'(t)
\end{equation*}
is strictly positive on $\mathopen{]}t_{0},t_{0}+\varepsilon\mathclose{[}$. We conclude that
$v(\cdot)$ is strictly increasing on a right neighborhood of $t_{0}$ and thus
$v(t) > v(t_{0})$ for all $t \in \mathopen{]}t_{0},t_{0}+\varepsilon\mathclose{[}$. This contradicts the fact that
$t_{0}$ is a maximum point for $v(\cdot)$.
The proof is thus complete.
\end{proof}

At last, by combining Theorem~\ref{th-bs} with Lemma~\ref{lem-bf1}, we are in position to state our main result which combines the abstract continuation theorem with the bounding function technique.

\begin{theorem}\label{th-bs2}
Let $F = F(t,x,y;\lambda) \colon \mathopen{[}0,T\mathclose{]} \times \mathbb{R}^{n}\times \mathbb{R}^{n}\times \mathopen{[}0,1\mathclose{]}\to \mathbb{R}^{n}$ be a continuous function such that
\begin{equation*}
F(t,x,y;1) = f(t,x,y), \qquad F(t,x,y;0) = f_{0}(x,y),
\end{equation*}
where $f_{0} \colon \mathbb{R}^{n}\times \mathbb{R}^{n} \to \mathbb{R}^{n}$
is an autonomous vector field. Suppose that there exists an open bounded set
$G\subseteq \mathbb{R}^{n}$ satisfying condition $(\textsc{h}_{\mathrm{N}})$. Assume that $G$ admits a family $(V_{u})_{u\in\partial G}$ of bounding functions of class $\mathcal{C}^2$ satisfying $(\textsc{h}_{V})$. Assume also $(\textsc{h}_{D})$. 
Then, problem \eqref{eq-phi}-\eqref{eq-per} has at least a solution $\tilde{x}$ such that $\tilde{x}(t)\in \overline{G}$, for all $t\in\mathopen{[}0,T\mathclose{]}$.
\end{theorem}

Using the continuation Theorem~\ref{th-cont-2} we can provide a variant of Theorem~\ref{th-bs2} which reads as follows.
Clearly the next result is strongly connected with the classical one in \cite{Ma-74} for $\phi(\xi)=\xi$.

\begin{theorem}\label{th-bs3}
Let $F = F(t,x,y) \colon \mathopen{[}0,T\mathclose{]} \times \mathbb{R}^{n}\times \mathbb{R}^{n}\to \mathbb{R}^{n}$ be a continuous function. Suppose that there exists an open bounded set
$G\subseteq \mathbb{R}^{n}$ satisfying condition $(\textsc{h}'_{\mathrm{N}})$. Assume that $G$ admits a family $(V_{u})_{u\in\partial G}$ of bounding functions of class $\mathcal{C}^2$ satisfying 
\begin{itemize}[leftmargin=30pt,labelsep=10pt]
\item [$(\textsc{h}'_{V})$] for every $u\in \partial G$, $t\in \mathopen{[}0,T\mathclose{]}$ and $\lambda\in \mathopen{]}0,1\mathclose{[}$
\begin{equation*}
\langle V''_{u}(u)y,\phi(y)\rangle + \lambda \langle V'_{u}(u),F(t,u,y)\rangle > 0,
\quad \forall\, y\in \mathbb{R}^{n} \colon \langle V'_{u}(u),y\rangle = 0.
\end{equation*}
\end{itemize}
Assume also $(\textsc{h}'_{\mathrm{D}})$.
Then, problem \eqref{eq-phi}-\eqref{eq-per} has at least a solution $\tilde{x}$ such that $\tilde{x}(t)\in \overline{G}$, for all $t\in\mathopen{[}0,T\mathclose{]}$.
\end{theorem}

Variants of the results presented in this section could be provided for the more general case of a vector field satisfying the Carath\'{e}odory conditions.
For the linear differential operator $\phi(\xi)=\xi$, this approach has been already considered in \cite{MaTh-11} (see also \cite{AnKoMa-09,AnMaPa-09,Ta-08,TaZa-07}).
In the Carath\'{e}odory case, pointwise estimates must be replaced by suitable control of the solutions in a neighborhood of the boundary of the set $G$. For simplicity in the exposition, we will not investigate this framework.

\section{Applications}\label{section-3}

In this section, we present some applications of our main results contained in Section~\ref{section-2}. In Section~\ref{section-3.1} we consider frictionless vector fields $f=f(t,x)$, while in Section~\ref{section-3.2} and Section~\ref{section-3.3} we analyse more general situations: a $\phi$-Laplacian Rayleigh-type equation and a $\phi$-Laplacian Li\'{e}nard equation, respectively.

\subsection{The case $f=f(t,x)$}\label{section-3.1}

We start with the simplest case $f=f(t,x)$, namely, we consider the periodic boundary value problem
\begin{equation}\label{eq-fx}
\begin{cases}
\, ( \phi(x') )' = f(t,x),\\
\, x(0) = x(T), \quad x'(0) = x'(T),
\end{cases}
\end{equation}
where $\phi$ is an homeomorphism of the form \eqref{phi-A} (considered in the previous section) and $f = f(t,x)\colon \mathopen{[}0,T\mathclose{]} \times \mathbb{R}^{n}\to \mathbb{R}^{n}$ is a continuous function.

In our first result, we apply the concept of field of \textit{outer normals} at the boundary of a convex body. 
The use of outer normals in the theory of bound sets is classical and dates back to the Seventies (cf.~\cite{Be-74,GuSc-74}). We refer to \cite{MaSD-17,MaSD-19,MaSD-19rimut} for a recent renewal of considerations of this technique for boundary value problems of non-local type and also for the rich bibliography.

Recall that if $D\subseteq \mathbb{R}^{n}$ is a convex body (i.e.~the closure of an open bounded convex set of $\mathbb{R}^{n}$)
then for each $u\in \partial D$ there exists a \textit{normal cone} $\mathcal{N}_{u}$ such that for each $\nu\in \mathcal{N}_{u}\setminus\{0\}$ we have that
\begin{equation*}
D\subseteq \bigl{\{} x\in \mathbb{R}^{n} \colon \langle x-u,\nu\rangle \leq 0 \bigr{\}}.
\end{equation*}
A family of vectors $\mathcal{V}=(\nu_{u})_{u\in\partial D}$ is called a field of \textit{outer normals}
for $D$ if $\nu_{u}\in \mathcal{N}_{u}\setminus\{0\}$ for each $u\in \partial D$ (cf.~\cite{FoGi-16}).
We notice that if $0\in\mathrm{int}D$ then we also have $\langle u, \nu_{u} \rangle >0$ for all $u\in\partial D$ with $\nu_{u}\in \mathcal{N}_{u}\setminus\{0\}$.

The following result holds.

\begin{corollary}\label{cor-3.1}
Let $\phi \colon \mathbb{R}^{n} \to \phi(\mathbb{R}^{n})=\mathbb{R}^{n}$
be a homeomorphism of the form \eqref{phi-A} and $f = f(t,x)\colon \mathopen{[}0,T\mathclose{]} \times \mathbb{R}^{n}\to \mathbb{R}^{n}$ be a continuous function.
Suppose that there exists a field of outer normals
${\mathcal V}=(\nu_{u})_{u\in\partial G}$, for
an open convex bounded set $G\subseteq\mathbb{R}^{n}$ such that
\begin{equation*}
\langle f(t,u),\nu_{u}\rangle \geq 0, \quad \text{for all $t\in \mathopen{[}0,T\mathclose{]}$.}
\end{equation*}
Then, problem \eqref{eq-fx} has at least a solution with values in
$\overline{G}$.
\end{corollary}

\begin{proof}
We fix a point $P\in G$ and observe that from basic properties of convex sets
(cf.~\cite[p.~107]{FeZa-88} and \cite{La-82})
we have $\langle u-P,\nu_{u}\rangle > 0$
for all $u\in \partial G$. Thus, we define
\begin{equation*}
F(t,x,y;\lambda):= \lambda f(t,x) + (1-\lambda)(x-P),
\end{equation*}
so that $F(t,x,y;1)=f(t,x)$ and $F(t,x,y;0)=f_{0}(x)= x-P$.
In this manner, $(\textsc{h}_{\mathrm{D}})$ is clearly satisfied as
\begin{equation*}
\mathrm{d}_{\mathrm{B}}(\mathrm{Id}_{\mathbb{R}^{n}}-P,G,0) = \mathrm{d}_{\mathrm{B}}(\mathrm{Id}_{\mathbb{R}^{n}},G,P)=1.
\end{equation*}
To prove condition $(\textsc{h}_{\mathrm{V}})$ we define $V_{u}(x):= \langle x-u,\nu_{u}\rangle$ and observe that
$V'_{u}(u) = \nu_{u}$ and $V''_{u}(u)=0$. Hence
\begin{equation*}
\langle V_{u}(u),F(t,u,y;\lambda)\rangle = 
\lambda \langle f(t,u),\nu_{u}\rangle + (1-\lambda)\langle u-P,\nu_{u}\rangle
\geq (1-\lambda)\langle u-P,\nu_{u}\rangle >0,
\end{equation*}
for every $\lambda\in \mathopen{[}0,1\mathclose{[}$, $t\in \mathopen{[}0,T\mathclose{]}$ and $u\in \partial G$; then $(\textsc{h}_{V})$
holds.

To conclude the proof we have still to check $(\textsc{h}_{N})$, that is
\eqref{eq-fx} is a Nagumo equation with respect to $G$. To this
end, let $\lambda\in\mathopen{[}0,1\mathclose{[}$ and $x(\cdot)$ be a
solution of $(P_{\lambda})$ with values in the bounded set $G$.
Then, $x(\cdot)$ is uniformly bounded and consequently
$f(\cdot,x(\cdot))$ is bounded as well. Hence, there are two
positive constants $R_{1}, R_{2}$ such that every solution of
$(P_{\lambda})$ satisfies
\begin{equation*}
\|x\|_{\infty}\leq R_{1},\quad \|(\phi(x'))'\|_{\infty}\leq R_{2}.
\end{equation*}
Let $\omega\in \mathbb{S}^{n-1}$, and consider
the auxiliary function $\psi_{x,\omega} (t):=\langle x(t),\omega\rangle$, $t\in\mathopen{[}0,T\mathclose{]}$. By Rolle's theorem there exists $t_0=t_0(x,\omega)\in
\mathopen{[}0,T\mathclose{]}$ such that $\psi_{x,\omega}'(t_0)=\langle x'(t_0),\omega\rangle = 0$. 
Therefore, for $z(t):=\phi(x'(t))=A(x'(t)) x'(t)$ the conditions in Lemma~\ref{lem-2.1} are satisfied with $M_{0}:=0$, $p=\infty$, $M_{1}:=R_{2}$. We conclude hat there exists a constant $K:= T R_{2}$ such that
\begin{equation*}
\|\phi(x'(t))\| \leq K, \quad \text{for every $t\in \mathopen{[}0,T\mathclose{]}$,}
\end{equation*}
holds for every possible solution $x(\cdot)$ of $(P_{\lambda})$.
As a consequence of the above inequality, we get that
$\phi(x'(\cdot))$, and so $x'(\cdot)$, is uniformly bounded. This proves $(\textsc{h}_{N}).$
Finally, we conclude with an application of Theorem~\ref{th-bs2}.
\end{proof}

In the next result the domain $G$ is the sublevel set of a functional. Preliminarily,
we make some remarks.

For given function $V \colon \mathbb{R}^{n}\to \mathbb{R}$ and $c\in\mathbb{R}$, we denote by $[V\leq c]$ the sublevel set $\{x\in\mathbb{R}^{n} \colon V(x)\leq c\}$ and similarly $[V=c]:=\{x\in\mathbb{R}^{n} \colon V(x)=c\}$ and $[V<c]:=\{x\in\mathbb{R}^{n} \colon V(x)<c\}$.
Assume that $V$ is of class $\mathcal{C}^{1}$ and that, for some $c\in\mathbb{R}$, the set $D:=[V\leq c]$ is nonempty and, moreover,
\begin{equation}\label{cond-V'}
V'(u)\neq 0, \quad \text{for all $u\in[V=c]$.}
\end{equation}
In this case, it turns out that $D$ is a regularly closed set, namely for $G:=\mathrm{int}D = D\setminus\partial D$ we have $\overline{G}=D$
and, moreover, $\partial G = \partial D = [V=c]$ as well as $G=[V<c]$.

To prove this claim, we observe that
in any case, we have
$[V<c]\subseteq G,$
$\overline{G}\subseteq D$
and $\partial G\subseteq \partial D \subseteq [V=c]$.
On the other hand, if for every $u\in \partial D$ we consider the
scalar function $w(\vartheta):= V(u+\vartheta V'(u)),$
we find that $w(0)=c$ and $w'(0) >0,$ so that in any neighborhood of $u$
there are points
of $G,$ from which we obtain that $\partial D \subseteq \overline{G}$.
Finally, from $\partial G=\partial D$ we also conclude that
$\overline{G} = G \cup \partial G= \mathrm{int}D\cup \partial D = D$
and thus the claim is proved.

Generally speaking, the condition $V'(u)\neq 0$, for all $u\in\partial D$, is not enough to guarantee the same property, as shown by the example $V\colon\mathbb{R}\to\mathbb{R}$, $V(x):=x^{4}-x^{2}$, and $c=0$. In this case, $\partial D = \{-1,1\} \subsetneq \{-1,0,1\}=[V=0]$.

The next lemma, which is borrowed and adapted from \cite{KrPa-99} (see also the survey \cite{Kr-1011}), is useful in this context as it allows to study the case of sublevel domains with a minimal set of assumptions.

\begin{lemma}\label{lem-convex}
Let $V \colon \mathbb{R}^{n}\to \mathbb{R}$ be a $\mathcal{C}^{2}$-function satisfying \eqref{cond-V'} for some $c\in\mathbb{R}$ with $D:=[V\leq c]$ nonempty, bounded and connected. Then, $D$ (as well as $\mathrm{int}D$) is convex if and only if 
\begin{itemize}[leftmargin=28pt,labelsep=10pt]
\item[$(C)$] for all $u\in\partial D$, $\langle V''(u)y,y\rangle \geq 0$,
for all $y\in \mathbb{R}^{n}$ with $\langle V'(u),y\rangle = 0$.
\end{itemize}
\end{lemma}

Lemma~\ref{lem-convex} can be proved by combining the results in \cite[ch.~6, pp.~195--196]{KrPa-99} or \cite{Kr-1011}.
In Appendix~\ref{appendix-A} we provide a sketch of the proof.
Clearly it is not restrictive to consider the case $c=0$, as done in the sequel.

The following result provides a variant of Corollary~\ref{cor-3.1} for convex sets with smooth boundary, observing that the involved domain is a sublevel set (see \cite{Be-74,Ma-74} for classical results in this direction for the linear differential operator).

\begin{corollary}\label{cor-V}
Let $\phi \colon \mathbb{R}^{n} \to \phi(\mathbb{R}^{n})=\mathbb{R}^{n}$
be a homeomorphism of the form \eqref{phi-A} and $f = f(t,x)\colon \mathopen{[}0,T\mathclose{]} \times \mathbb{R}^{n}\to \mathbb{R}^{n}$ be a continuous function.
Let $V \colon \mathbb{R}^{n}\to \mathbb{R}$ be a $\mathcal{C}^{2}$-function such that $V'(u)\neq 0$ for all $u\in[V=0]$.
Let $D:=[V\leq0]$ be nonempty and bounded.
Suppose that for every $t\in \mathopen{[}0,T\mathclose{]}$ and $u\in \partial D$
\begin{itemize}[leftmargin=28pt,labelsep=10pt]
\item $\langle V'(u),f(t,u)\rangle \geq 0$;
\item $\langle V''(u)y,y\rangle \geq 0$,
for all $y\in \mathbb{R}^{n}$ with $\langle V'(u),y\rangle = 0$.
\end{itemize}
Then, problem \eqref{eq-fx} has at least a solution with values in $D$.
\end{corollary}

\begin{proof}
We define $G:={\rm int}D = D\setminus\partial D$ and therefore $G=[V<0]$.
By the preliminary observation we also have
$\overline{G}= D$
and $\partial G=\partial D = [V=0]$.
Let $G_{0}$ be a connected component of $G$. Setting $V_{u}=V$ for all $u\in\partial G_{0}$,
we have that $(V_{u})_{u\in\partial G_{0}}$ is a family of bounding functions for $G_{0}$.
Next, we consider the homotopy
\begin{equation*}
F(t,x,y;\lambda) := \lambda f(t,x) + (1-\lambda) V'(x)
\end{equation*}
and observe that
\begin{align*}
&\langle V''(u)y,\phi(y)\rangle + \langle V'(u),F(t,u,y;\lambda)\rangle =
\\&= A(y) \langle V''(u)y,y\rangle
+ \lambda \langle V'(u),f(t,u)\rangle
+ (1-\lambda) \|V'(u)\| >0,
\end{align*}
for every $\lambda\in \mathopen{[}0,1\mathclose{[}$, $t\in \mathopen{[}0,T\mathclose{]}$, $u\in \partial G_{0}$, and $y\in \mathbb{R}^{n}$ such that $\langle V'(u),y\rangle = 0$.
Therefore, condition $(\textsc{h}_{V})$ holds.
Next, from Lemma~\ref{lem-convex} we observe that $G_{0}$ is convex and therefore
\begin{equation}\label{deg-1}
\mathrm{d}_{\mathrm{B}}(V',G_{0},0) = 1.
\end{equation}
To prove the above formula, let us fix a point $P_{0}\in G_{0}$. We claim that 
\begin{equation}\label{grad-P0}
\langle V'(u),u-P_{0} \rangle \geq 0, \quad \text{for all $u\in\partial G_{0}$.}
\end{equation}
Indeed, setting $v(\vartheta):= V(P_{0}+\vartheta(u-P_{0}))$ for $\vartheta\in\mathopen{[}0,1\mathclose{]}$, we have that $v(1)=0$ and $v(\vartheta)<0$ for all $\vartheta\in\mathopen{[}0,1\mathclose{[}$. Hence, $\langle V'(u),u-P_{0} \rangle = v'(1) \geq 0$. From \eqref{grad-P0}, via the convex homotopy $(1-\lambda) V'(x) + \lambda (x-P_{0})$ we find that
\begin{equation*}
\mathrm{d}_{\mathrm{B}}(V',G_{0},0) = \mathrm{d}_{\mathrm{B}}(\mathrm{Id}_{\mathbb{R}^{n}}-P_{0},G_{0},0)=\mathrm{d}_{\mathrm{B}}(\mathrm{Id}_{\mathbb{R}^{n}},G_{0},P_{0})=1.
\end{equation*}
An alternative manner to prove \eqref{deg-1} is to apply \cite[Theorem~3]{FoGi-16} to the vector field $-V'(x)$ on the convex set $G_{0}$.

Finally, we conclude with an application of Theorem~\ref{th-bs2} which implies the existence of a solution in $\overline{G_{0}}$ and hence in $D$.
\end{proof}

We can reestablish the following version of the Hartman--Knobloch theorem (see \cite{FeZa-17,Ma-00}).

\begin{theorem}[Hartman--Knobloch]\label{th-HK}
Let $\phi \colon \mathbb{R}^{n} \to \phi(\mathbb{R}^{n})=\mathbb{R}^{n}$
be a homeomorphism of the form \eqref{phi-A} and $f = f(t,x)\colon \mathopen{[}0,T\mathclose{]} \times \mathbb{R}^{n}\to \mathbb{R}^{n}$ be a continuous function. If there exists $R > 0$ such that Hartman's condition
\begin{equation*}
\langle f(t,x),x\rangle \geq 0, \quad \text{for every $t\in\mathopen{[}0,T\mathclose{]}$ and $x\in\mathbb{R}^{n}$ with $\|x\|= R$,}
\end{equation*}
holds, then there exists a solution $x$ of \eqref{eq-fx} such that $\|x(t)\|\leq R$ for all $t\in\mathopen{[}0,T\mathclose{]}$.
\end{theorem}

\begin{proof}
We can propose two different proofs. The first one is based on Corollary~\ref{cor-3.1}. We consider the set $G=\{x\in\mathbb{R}^{n}\colon \|x\|<R\}$ and observe that $\nu_{x}=x$ for all $x\in\partial G$. Therefore, Hartman's condition corresponds to the assumption in Corollary~\ref{cor-3.1} and thus Theorem~\ref{th-HK} follows.

The second proof is based on Corollary~\ref{cor-V}. We deal with the $\mathcal{C}^{\infty}$-function $V(x)=(\|x\|^{2}-R^{2})/2$ and notice that $V'(x)=x$ and $V''$ is the identity matrix.
The conclusion is reached as a direct application of Corollary~\ref{cor-V}.
\end{proof}

We can further state another straightforward consequence of Corollary~\ref{cor-3.1}.

\begin{theorem}[Poincar\'{e}--Miranda]\label{th-PM}
Let $\phi \colon \mathbb{R}^{n} \to \phi(\mathbb{R}^{n})=\mathbb{R}^{n}$
be a homeomorphism of the form \eqref{phi-A} and let $f = f(t,x)\colon \mathopen{[}0,T\mathclose{]} \times \mathbb{R}^{n}\to \mathbb{R}^{n}$ be a continuous function. Assume that there exists a $n$-dimensional rectangle $\mathcal{R}:=\prod_{i=1}^{n} \mathopen{[}a_{i},b_{i}\mathclose{]}$ such that for each $i\in\{1,\ldots,n\}$ it holds that
\begin{equation}\label{cond-PM}
\begin{cases}
\, f_{i}(t,x)\leq 0, & \text{for all $x\in\partial\mathcal{R}$ with $x_{i}=a_{i}$,}
\\
\, f_{i}(t,x)\geq 0, & \text{for all $x\in\partial\mathcal{R}$ with $x_{i}=b_{i}$.}
\end{cases}
\end{equation}
Then, there exists a solution $x$ of \eqref{eq-fx} with $x(t)\in\mathcal{R}$ for all $t\in\mathopen{[}0,T\mathclose{]}$.
\end{theorem}

As is well known, the inequalities in \eqref{cond-PM} can be reversed for first order systems; on the other hand, for second-order systems (as in our case) reversing  the inequalities is not possible unless further growth (non-resonance) assumptions on the vector field are imposed, as one can see from the trivial example $x''=-x+\varepsilon \sin(t)$ (see, for instance, \cite{Ma-72} for a result in this direction).

It is worth noticing that Theorem~\ref{th-PM} is related to the theory of (well ordered) lower and upper solutions for a differential system in $\mathbb{R}^{n}$,
as it represents the case of vector valued lower solution $\alpha$ and vector valued upper solution $\beta$ with $\alpha_{i}\equiv a_{i}$ and $\beta_{i}=b_{i}$.
In order to deal with non-constant lower and upper solutions in $\mathbb{R}^{n}$, one could adapt to our case the technique of ``curvature bound sets'' in \cite{GaMa-77jde,GaMa-77} or ``non autonomous bounding Lyapunov functions'' in \cite{Be-74}; see also \cite{HaSc-83} for a similar approach.

\subsection{A $\phi$-Laplacian Rayleigh-type equation}\label{section-3.2}

Up to now we have considered only a trivial case of the Nagumo--Hartman condition, namely when the vector field does not depend on $x'$. In the next example, inspired by a case studied in \cite{Ma-74}, we treat a more general situation.

Let us consider the following $T$-periodic boundary value problem associated with a $\phi$-Laplacian Rayleigh-type equation
\begin{equation}\label{eq-R}
\begin{cases}
\, ( \phi(x') )' = g(x') + h(t,x),\\
\, x(0) = x(T), \quad x'(0) = x'(T),
\end{cases}
\end{equation}
which is a natural generalisation of the system $x'' = g(x') + h(t,x)$ considered in \cite{Ma-74}.
In this framework, we state the following result, which is in the spirit of Corollary~\ref{cor-V}. It can be seen as an extension of \cite[Corollary~6.3]{Ma-74} (for $V(x)=(\|x\|^{2}-R^{2})/2$).

\begin{theorem}\label{th-Rayleigh}
Let $\phi \colon \mathbb{R}^{n} \to \phi(\mathbb{R}^{n})=\mathbb{R}^{n}$ be a homeomorphism of the form \eqref{phi-A} and let
$h = h(t,x)\colon \mathopen{[}0,T\mathclose{]} \times \mathbb{R}^{n}\to \mathbb{R}^{n}$ be a continuous function.
Let $g\colon \mathbb{R}^{n}\to \mathbb{R}^{n}$ be continuous and such that $g(y)$ has either the same or the opposite direction of $y$.
Assume also that there exists $\mathcal{G}\colon\mathbb{R}^{n}\to \mathbb{R}$ of class $\mathcal{C}^{1}$ such that 
\begin{itemize}[leftmargin=28pt,labelsep=10pt]
\item $g(y)-\nabla \mathcal{G}(\phi(y))$ is bounded.
\end{itemize}
Let $V \colon \mathbb{R}^{n}\to \mathbb{R}$ be a $\mathcal{C}^{2}$-function such that $V'(u)\neq 0$ for all $u\in[V=0]$.
Let $D:=[V\leq0]$ be nonempty and bounded.
Suppose that for every $t\in \mathopen{[}0,T\mathclose{]}$ and $u\in \partial D$
\begin{itemize}[leftmargin=28pt,labelsep=10pt]
\item $\langle V'(u),h(t,u)\rangle \geq 0$;
\item $\langle V''(u)y,y\rangle \geq 0$,
for all $y\in \mathbb{R}^{n}$ with $\langle V'(u),y\rangle = 0$.
\end{itemize}
Then, problem \eqref{eq-R} has at least a solution with values in $D$.
\end{theorem}

\begin{proof}
The proof follows the same steps of Corollary~\ref{cor-V}. Accordingly we just point out the main modifications which are needed.
Having defined $G:={\rm int}D =[V<0]$ and $G_{0}$ as above, we consider the homotopy
\begin{equation*}
F(t,x,y;\lambda) := g(y) + \lambda h(t,x) + (1-\lambda) V'(x).
\end{equation*}
For every $\lambda\in \mathopen{[}0,1\mathclose{[}$, $t\in \mathopen{[}0,T\mathclose{]}$, $u\in \partial G_{0}$, and $y\in \mathbb{R}^{n}$ such that $\langle V'_{u}(u),y\rangle = 0$, it holds that
\begin{align*}
&\langle V''(u)y,\phi(y)\rangle + \langle V'(u),F(t,u,y;\lambda)\rangle =
\\&= A(y) \langle V''(u)y,y\rangle + \langle V'(u),g(y) \rangle 
+ \lambda \langle V'(u),h(t,u)\rangle
+ (1-\lambda) \|V'(u)\|
\\&=A(y) \langle V''(u)y,y\rangle + \lambda \langle V'(u),h(t,u)\rangle
+ (1-\lambda) \|V'(u)\|>0,
\end{align*}
where we have used the fact that $g(y)$ is parallel to $y$. Therefore, condition $(\textsc{h}_{V})$ holds. 
The verification of the degree conditions is the same as in the proof of Corollary~\ref{cor-V} using the convexity of $G_{0}$.

In order to conclude the proof, we need only to check that the Nagumo condition $(\textsc{h}_{N})$ is satisfied with respect to the set $G_{0}$.
To this end, let $\lambda\in\mathopen{[}0,1\mathclose{[}$ and $x(\cdot)$ be a
solution of $(P_{\lambda})$ with values in the bounded set $G_{0}$.
Then, $x(\cdot)$ is uniformly bounded and consequently
\begin{equation*}
w(t) := \lambda h(t,x(t)) + (1-\lambda) V'(x(t))
\end{equation*}
is uniformly bounded too, by a constant $R_{0}$.
Consider now the system
\begin{equation*}
\begin{cases}
\, ( \phi(x') )' = g(x') + w(t),\\
\, x(0) = x(T), \quad x'(0) = x'(T).
\end{cases}
\end{equation*}
From this we have
\begin{align*}
\| (\phi(x'(t)))' \|^{2} &= \langle g(x'(t))-\nabla \mathcal{G}(\phi(x'(t))), (\phi(x'(t)))' \rangle 
\\
&\quad + \langle \nabla \mathcal{G}(\phi(x'(t))), (\phi(x'(t)))' \rangle + \langle w(t), (\phi(x'(t)))' \rangle
\\
&\leq L \|(\phi(x'(t)))'\| + R_{0} \|(\phi(x'(t)))'\| + \dfrac{\mathrm{d}}{\mathrm{d}t} \mathcal{G}(\phi(x'(t))),
\end{align*}
where $L>0$ is a constant which bounds $g(y)-\nabla \mathcal{G}(\phi(y))$.
An integration on $\mathopen{[}0,T\mathclose{]}$ yields to
\begin{equation*}
\| (\phi(x'))' \|_{L^{2}}^{2} \leq (L+R_{0}) \sqrt{T} \, \|(\phi(x'))'\|_{L^{2}}
\end{equation*}
and thus a bound for $(\phi(x'(\cdot)))'$ in $L^{2}$ is achieved.
An application of Lemma~\ref{lem-2.1} for $p=2$, arguing as in the proof of Corollary~\ref{cor-3.1}, ensures an a priori bound for $\|x'\|_{\infty}$.
Finally, Theorem~\ref{th-bs2} implies the existence of a solution in $\overline{G_{0}}$ and hence in $D$.
\end{proof}

As in Section~\ref{section-3.1} we have the following straightforward corollaries.

\begin{corollary}[Hartman--Knobloch]\label{th-HK-2}
Let $\phi \colon \mathbb{R}^{n} \to \phi(\mathbb{R}^{n})=\mathbb{R}^{n}$ be a homeomorphism of the form \eqref{phi-A} and let
$h = h(t,x)\colon \mathopen{[}0,T\mathclose{]} \times \mathbb{R}^{n}\to \mathbb{R}^{n}$ be a continuous function.
Let $g\colon \mathbb{R}^{n}\to \mathbb{R}^{n}$ be continuous and such that $g(y)$ has either the same or the opposite direction of $y$.
If there exists $R > 0$ such that Hartman's condition
\begin{equation*}
\langle h(t,x),x\rangle \geq 0, \quad \text{for every $t\in\mathopen{[}0,T\mathclose{]}$ and $x\in\mathbb{R}^{n}$ with $\|x\|= R$,}
\end{equation*}
holds, then there exists a solution $x$ of \eqref{eq-fx} such that $\|x(t)\|\leq R$ for all $t\in\mathopen{[}0,T\mathclose{]}$.
\end{corollary}

\begin{corollary}[Poincar\'{e}--Miranda]\label{th-PM-2}
Let $\phi \colon \mathbb{R}^{n} \to \phi(\mathbb{R}^{n})=\mathbb{R}^{n}$
be a homeomorphism of the form \eqref{phi-A} and let $h = h(t,x)\colon \mathopen{[}0,T\mathclose{]} \times \mathbb{R}^{n}\to \mathbb{R}^{n}$ be a continuous function. Assume that there exists a $n$-dimensional rectangle $\mathcal{R}:=\prod_{i=1}^{n} \mathopen{[}a_{i},b_{i}\mathclose{]}$ such that for each $i\in\{1,\ldots,n\}$ it holds that
\begin{equation*}
\begin{cases}
\, h_{i}(t,x)\leq 0, & \text{for all $x\in\partial\mathcal{R}$ with $x_{i}=a_{i}$,}
\\
\, h_{i}(t,x)\geq 0, & \text{for all $x\in\partial\mathcal{R}$ with $x_{i}=b_{i}$.}
\end{cases}
\end{equation*}
Then, there exists a solution $x$ of \eqref{eq-fx} with $x(t)\in\mathcal{R}$ for all $t\in\mathopen{[}0,T\mathclose{]}$.
\end{corollary}

\subsection{A $\phi$-Laplacian Li\'{e}nard equation}\label{section-3.3}

All the preceding examined examples depend on the fact that first we find a set $G\subseteq\mathbb{R}^{n}$ with no solutions tangent to the boundary from the interior, and next we provide some a priori bound on $\|x'(t)\|$ using Nagumo-type conditions. There are however some situations in which the particular form of the equations allows to find a priori bounds on $\|x'\|_{L^{p}}$ (for some $p$) independently on $x$. In such cases, the set $G$ can be found using some ``sign-conditions'' on the nonlinearity. As a possible example in this direction and our third application, we deal with the following $T$-periodic boundary value problem associated with a $\phi$-Laplacian Li\'{e}nard equation
\begin{equation}\label{eq-L-0}
\begin{cases}
\, ( \phi(x') )' = \dfrac{\mathrm{d}}{\mathrm{d}t} \nabla \mathcal{G}(x) + h(t,x),\\
\, x(0) = x(T), \quad x'(0) = x'(T).
\end{cases}
\end{equation}
In this setting, we can state the following.

\begin{theorem}\label{th-Lienard-0}
Let $\phi \colon \mathbb{R}^{n} \to \phi(\mathbb{R}^{n})=\mathbb{R}^{n}$ be a homeomorphism such that
\begin{itemize}[leftmargin=28pt,labelsep=10pt]
\item [$(\textsc{h}_{\phi})$] $\langle \phi(\xi),\xi \rangle >0$ for every $\xi\in\mathbb{R}^{n}\setminus\{0\}$, and for every $\eta>0$ there exists $M_{\eta}>0$ such that $\langle \phi(\xi),\xi \rangle \geq \eta \|\xi\| - M_{\eta}$, for all $\xi\in\mathbb{R}^{n}$.
\end{itemize}
Let $\mathcal{G}\colon\mathbb{R}^{n}\to\mathbb{R}^{n}$ be a $\mathcal{C}^{2}$-function. Let $h=h(t,x) \colon \mathopen{[}0,T\mathclose{]}\times\mathbb{R}^{n}\to \mathbb{R}^{n}$ be a continuous function such that
\begin{itemize}[leftmargin=28pt,labelsep=10pt]
\item [$(\textsc{h}_{\textsc{h}})$] there exists $R > 0$ such that $\langle h(t,x),x\rangle \geq 0$, for every $t\in\mathopen{[}0,T\mathclose{]}$ and $x\in\mathbb{R}^{n}$ with $\|x\| \geq R$.
\end{itemize}
Then, problem \eqref{eq-L-0} has at least a solution.
\end{theorem}

\begin{proof}
We aim to apply Theorem~\ref{th-cont}. Accordingly, we introduce the parameter-dependent problem
\begin{equation}\label{eq-L-0-lambda}
\begin{cases}
\, ( \phi(x') )' = \lambda \dfrac{\mathrm{d}}{\mathrm{d}t} \nabla \mathcal{G}(x) + h_{\lambda}(t,x),\\
\, x(0) = x(T), \quad x'(0) = x'(T),
\end{cases}
\end{equation}
where
\begin{equation}\label{h-lambda}
h_{\lambda}(t,x) := \lambda h(t,x) + (1-\lambda) x, \quad \lambda\in\mathopen{[}0,1\mathclose{]}.
\end{equation}
We divide the proof in some steps.

\smallskip
\noindent
\textit{Step~1. A priori bound of $\|x'\|_{L^{1}}$.} Let $x(\cdot)$ be a solution of problem~\eqref{eq-L-lambda}. By integrating the scalar product between $x(\cdot)$ and the equation in \eqref{eq-L-0-lambda}, we have
\begin{equation}\label{eq-3.6}
\int_{0}^{T} \langle \phi(x'(t)),x'(t) \rangle \,\mathrm{d}t 
+  \int_{0}^{T}\langle h_{\lambda}(t,x(t)),x(t)\rangle\,\mathrm{d}t 
=0.
\end{equation}
Next, we observe that condition $(\textsc{h}_{\textsc{h}})$ implies the existence of $K_{0}>0$ such that $\langle h(t,x),x\rangle \geq -K_{0}$, for every $t\in\mathopen{[}0,T\mathclose{]}$ and $x\in\mathbb{R}^{n}$. Therefore, from \eqref{eq-3.6}, we deduce that
\begin{equation*}
\int_{0}^{T} \langle \phi(x'(t)),x'(t) \rangle \,\mathrm{d}t 
\leq TK_{0}.
\end{equation*}
We fix $\eta=1$ and, by hypothesis $(\textsc{h}_{\phi})$, we obtain that
\begin{equation*}
\|x'\|_{L^{1}} = \int_{0}^{T} \|x'(t)\|\,\mathrm{d}t \leq T(K_{0}+M_{1})=:K_{1},
\end{equation*}
which is the desired bound.

\smallskip
\noindent
\textit{Step~2. A priori bound of $\|x\|_{\infty}$.}
Let $x(\cdot)$ be a solution of problem~\eqref{eq-L-lambda}.
First, we prove that there exists $\tilde{t}\in\mathopen{[}0,T\mathclose{]}$ such that $\|x(\tilde{t})\|<R$. Indeed, if it is not true, $\|x(t)\|\geq R$ for all $t\in\mathopen{[}0,T\mathclose{]}$ and from \eqref{eq-3.6} (and since $\lambda\in\mathopen{[}0,1\mathclose{[}$), we deduce that
\begin{equation*}
0 = \int_{0}^{T} \langle \phi(x'(t)),x'(t) \rangle \,\mathrm{d}t 
+ \lambda \int_{0}^{T}\langle h(t,x(t)),x(t)\rangle\,\mathrm{d}t + (1-\lambda) \int_{0}^{T} \|x(t)\|^{2}\,\mathrm{d}t
> 0,
\end{equation*}
a contradiction. Consequently, for every $t\in\mathopen{[}0,T\mathclose{]}$, we immediately obtain
\begin{equation*}
\|x(t)\|= \biggl{\|} x(\tilde{t}) + \int_{\tilde{t}}^{t} x'(s)\,\mathrm{d}s \biggr{\|}
\leq \|x(\tilde{t})\| + \int_{0}^{T}\|x'(s)\|\,\mathrm{d}s < R+K_{1}T=:R^{*},
\end{equation*}
and thus $\|x\|_{\infty}<R^{*}$, as desired. As a consequence, the open ball $G=B(0,R^{*})$ is (trivially) a bound set for system \eqref{eq-L-0-lambda}.

\smallskip
\noindent
\textit{Step~3. Conclusion.}
From
\begin{equation*}
\|(\phi(x'))'\|_{L^{1}}
= \int_{0}^{T} \|(\phi(x'(t)))'\|\,\mathrm{d}t
\leq \|\mathrm{Hess}\,\mathcal{G}(x)\|_{\infty} \|x'\|_{L^{1}}+ T \|h_{\lambda}(\cdot,x)\|_{\infty},
\end{equation*}
we have that $(\phi(x'(\cdot)))'$ is bounded in $L^{1}$.
Therefore, an application of Lemma~\ref{lem-2.1} for $p=1$, arguing as in the proof of Corollary~\ref{cor-3.1}, ensures an a priori bound for $\|x'\|_{\infty}$.

For $\lambda=0$, we have $\mathrm{d}_{\mathrm{B}}(\mathrm{Id}_{\mathbb{R}^{n}}, G, 0)= 1 \neq0$. Therefore, the thesis follows from Theorem~\ref{th-bs}.
\end{proof}

\begin{remark}\label{rem-3.1}
If $\phi \colon \mathbb{R}^{n} \to \phi(\mathbb{R}^{n})=\mathbb{R}^{n}$ is a homeomorphism of the form \eqref{phi-A}, then $\phi$ satisfies hypothesis $(\textsc{h}_{\phi})$ of Theorem~\ref{th-Lienard-0}. Indeed, for $\xi\neq0$, we have 
\begin{equation*}
\dfrac{\langle \phi(\xi),\xi\rangle}{\|\xi\|} = \dfrac{A(\xi)\|\xi\|^{2}}{\|\xi\|} = \|A(\xi)\xi\| = \|\phi(\xi)\| \to + \infty, \quad \text{as $\|\xi\|\to+\infty$.}
\end{equation*}
Hence, $(\textsc{h}_{\phi})$ follows.
\hfill$\lhd$
\end{remark}

Our result is related to a classical theorem by Reissig \cite[Theorem~3]{Re-75} for the classical scalar generalized Li\'{e}nard equation with a periodic forcing term
\begin{equation*}
x'' + \mathscr{F}(x) x' + \mathscr{G}(x) = p(t) \equiv p(t+T),
\end{equation*}
where the existence of a $T$-periodic solution is proved by assuming $\mathscr{G}(x)x\leq 0$ for every $x\in\mathbb{R}$ with $|x|\geq d>0$ and $\int_{0}^{T} p(t)\,\mathrm{d}t=0$. No special assumption besides continuity on $\mathscr{F}$ is considered.
Extensions of this and other related results for a $\phi$-Laplacian differential operator with or without singularity (including the Minkowski operator for the relativistic acceleration) have been obtained in \cite{BeMa-07,Ma-13} in the scalar case.
We show now how the proof of Theorem~\ref{th-Lienard-0} can be easily adapted to treat the case of a periodic forcing term with zero mean value, namely we deal with
\begin{equation}\label{eq-L}
\begin{cases}
\, ( \phi(x') )' = \dfrac{\mathrm{d}}{\mathrm{d}t} \nabla \mathcal{G}(x) + h(t,x) + p(t),\\
\, x(0) = x(T), \quad x'(0) = x'(T).
\end{cases}
\end{equation}

\begin{theorem}\label{th-Lienard}
Let $\phi \colon \mathbb{R}^{n} \to \phi(\mathbb{R}^{n})=\mathbb{R}^{n}$ be a homeomorphism satisfying $(\textsc{h}_{\phi})$. Let $\mathcal{G}\colon\mathbb{R}^{n}\to\mathbb{R}^{n}$ be a $\mathcal{C}^{2}$-function. Let $p \colon \mathopen{[}0,T\mathclose{]} \to \mathbb{R}^{n}$ be a continuous function with $\int_{0}^{T} p(t)\,\mathrm{d}t=0$.
Let $h=h(t,x) \colon \mathopen{[}0,T\mathclose{]}\times\mathbb{R}^{n}\to \mathbb{R}^{n}$ be a continuous function such that
\begin{itemize}[leftmargin=30pt,labelsep=10pt]
\item [$(\textsc{h}_{\textsc{h}}^{+})$] there exists $K_{0}>0$ such that $\langle h(t,x),x\rangle \geq -K_{0}$, for all $x\in\mathbb{R}^{n}$.
\end{itemize}
Moreover, suppose that at least one of the following three condition holds:
\begin{itemize}[leftmargin=28pt,labelsep=10pt]
\item [$(i)$] $\langle h(t,x), x \rangle \to +\infty$ as $\|x\|\to+\infty$ uniformly in $t$;
\item [$(ii)$] for every $\rho>0$ there exists $R_{\rho}>0$ such that $\langle h(x+y),x \rangle \geq 0$, for all $(x,y)\in\mathbb{R}^{2n}$ with $\|x\|>R_{\rho}$ and $\|y\|\leq\rho$;
\item [$(iii)$] there exists $d>0$ such that, for every $i=1,\ldots,n$, $h_{i}(t,x) x_{i} \geq 0$, for every $x=(x_{1},\ldots,x_{n})\in\mathbb{R}^{n}$ with $|x_{i}|> d$.
\end{itemize}
Then, problem \eqref{eq-L} has at least a solution.
\end{theorem}

\begin{proof}
The proof is similar to the one of Theorem~\ref{th-Lienard-0} and thus we only focus on the main modifications requested. We introduce the parameter-dependent problem
\begin{equation}\label{eq-L-lambda}
\begin{cases}
\, ( \phi(x') )' = \lambda \dfrac{\mathrm{d}}{\mathrm{d}t} \nabla \mathcal{G}(x) + h_{\lambda}(t,x) + \lambda p(t),\\
\, x(0) = x(T), \quad x'(0) = x'(T),
\end{cases}
\end{equation}
where $h_{\lambda}(t,x)$ is defined as in \eqref{h-lambda}.
Since $\int_{0}^{T} p(t)\,\mathrm{d}t=0$, it is convenient to fix a $T$-periodic continuously differentiable function $P\colon\mathbb{R}\to\mathbb{R}^{N}$ such that $P'(t)=p(t)$.
We divide the proof in some steps.

\smallskip
\noindent
\textit{Step~1. A priori bound of $\|x'\|_{L^{1}}$.} Let $x(\cdot)$ be a solution of problem~\eqref{eq-L-lambda}. By integrating the scalar product between $x(\cdot)$ and the equation in \eqref{eq-L-lambda}, we have
\begin{equation}\label{eq-3.10}
\begin{aligned}
\int_{0}^{T} \langle \phi(x'(t)),x'(t) \rangle \,\mathrm{d}t
&= - \int_{0}^{T}\langle h_{\lambda}(x(t)),x(t)\rangle\,\mathrm{d}t 
- \lambda \int_{0}^{T}\langle p(t),x(t)\rangle\,\mathrm{d}t .
\\
&= - \int_{0}^{T}\langle h_{\lambda}(x(t)),x(t)\rangle\,\mathrm{d}t 
+ \lambda \int_{0}^{T}\langle P(t),x'(t)\rangle\,\mathrm{d}t.
\end{aligned}
\end{equation}
Next, from hypothesis $(\textsc{h}_{\textsc{h}}^{+})$ and the fact that $\lambda\in\mathopen{[}0,1\mathclose{]}$, we deduce that
\begin{equation*}
\int_{0}^{T} \langle \phi(x'(t)),x'(t) \rangle \,\mathrm{d}t 
\leq TK_{0} + \|P\|_{\infty} \int_{0}^{T} \|x'(t)\|\,\mathrm{d}t.
\end{equation*}
We fix $\eta>\|P\|_{\infty}$ and, by hypothesis $(\textsc{h}_{\phi})$, we obtain that
\begin{equation*}
(\eta-\|P\|_{\infty}) \int_{0}^{T} \|x'(t)\|\,\mathrm{d}t \leq T(K_{0}+M),
\end{equation*}
and so
\begin{equation*}
\|x'\|_{L^{1}} \leq \dfrac{T(K_{0}+M)}{\eta-\|P\|_{\infty}}=:K_{1},
\end{equation*}
which is the desired bound.

\smallskip
\noindent
\textit{Step~2. A priori bound of $\|x\|_{\infty}$.} 
Without loss of generality, we can consider $\lambda\in\mathopen{]}0,1\mathclose{[}$; indeed, an easy computation shows that for $\lambda=0$ the only $T$-periodic solution is the trivial one. Let $x(\cdot)$ be a solution of problem~\eqref{eq-L-lambda}.

Assume condition $(i)$. 
From \eqref{eq-3.10} we have
\begin{equation*}
\lambda \int_{0}^{T}\langle h(t,x(t)),x(t)\rangle\,\mathrm{d}t +(1-\lambda)\int_{0}^{T}\|x(t)\|\mathrm{d}t \leq \lambda \|P\|_{\infty} K_{1},
\end{equation*}
and so, dividing by $\lambda>0$,
\begin{equation*}
\int_{0}^{T}\langle h(t,x(t)),x(t)\rangle\,\mathrm{d}t \leq \|P\|_{\infty} K_{1}.
\end{equation*}
Let $\gamma > \|P\|_{\infty} K_{1}/T$. By $(i)$ there exists $R_{\gamma}>0$ such that $\langle h(t,x(t)),x(t)\rangle > \gamma$ for all $t\in\mathopen{[}0,T\mathclose{]}$ such that $\|x(t)\|\geq R_{\gamma}$. We immediately conclude that there exists $\tilde{t}\in\mathopen{[}0,T\mathclose{]}$ such that $\|x(\tilde{t})\|<R_{\gamma}$, otherwise a contradiction can be easily obtained. Then, arguing as in the proof of Theorem~\ref{th-Lienard-0}, we have that $\|x\|_{\infty}< R_{\gamma}+K_{1}T:=R^{*}$.

Assume now condition $(ii)$ or condition $(iii)$. Let us also denote by $\bar{x}$ the mean value of $x(\cdot)$.
Let $\nu\in\mathbb{S}^{n-1}$ (arbitrary). An integration of the scalar product between $\nu$ and the equation in \eqref{eq-L-lambda} gives
\begin{equation*}
\int_{0}^{T} \langle h_{\lambda}(t,x(t)),\nu\rangle\,\mathrm{d}t = \lambda \int_{0}^{T} \langle p(t) , \nu\rangle\, \mathrm{d}t
= \lambda \langle \int_{0}^{T} p(t) \mathrm{d}t, \nu\rangle = 0,
\end{equation*}
which equivalently reads as
\begin{equation}\label{eq-nu}
\lambda \dfrac{1}{T}\int_{0}^{T} \langle h(t,x(t)),\nu\rangle\,\mathrm{d}t + (1-\lambda) \langle\bar{x},\nu\rangle =0.
\end{equation}
Let $(ii)$ hold. If $\bar{x}\neq0$, we set $\nu:=\bar{x}/\|\bar{x}\|$ and so $\int_{0}^{T} \langle h(t,x(t)),\bar{x}\rangle\,\mathrm{d}t <0$.
Then, there exists $\tilde{t}_{\bar{x}}\in\mathopen{[}0,T\mathclose{]}$ such that  $\langle h(\tilde{t}_{\bar{x}},x(\tilde{t}_{\bar{x}})),\bar{x}\rangle<0$.
On the other hand, we have that $\|x(\tilde{t}_{\bar{x}})-\bar{x}\|$ is bounded. Indeed, let $\hat{t}_{i}\in\mathopen{[}0,T\mathclose{]}$ be such that $x_{i}(\hat{t}_{i})=\bar{x}_{i}$, thus
\begin{equation*}
|x_{i}(\tilde{t}_{\bar{x}})-\bar{x}_{i}| = \biggl{|} \int_{\tilde{t}_{\bar{x}}}^{\hat{t}_{i}} x_{i}'(s)\,\mathrm{d}s \biggr{|} \leq \int_{0}^{T} |x_{i}'(s)|\,\mathrm{d}s \leq \int_{0}^{T} \|x'(s)\|\,\mathrm{d}s =\|x'\|_{L^{1}} \leq K_{1},
\end{equation*}
and so $\|x(\tilde{t}_{\bar{x}})-\bar{x}\|\leq \sqrt{n} \, K_{1}$. Therefore, from $x(\tilde{t}_{\bar{x}})=\bar{x} +(x(\tilde{t}_{\bar{x}})-\bar{x})$ and $(ii)$, we have that there exists $\hat{R}>0$ such that $\|\bar{x}\|< \hat{R}$. Clearly, the same inequality holds if $\bar{x}=0$.
Next, it is easy to prove that
\begin{equation*}
|x_{i}(t)| \leq |\bar{x}_{i}| + |x_{i}(t)-\bar{x}_{i}| \leq  \|\bar{x}\| + \|x'\|_{L^{1}} < \hat{R} + K_{1} 
\end{equation*}
and thus $\|x\|_{\infty} < \sqrt{n} (\hat{R} + K_{1}) =:R^{*}$.
Let $(iii)$ hold. We prove that for every $i=1,\ldots,n$ there exists $\tilde{t}_{i}\in\mathopen{[}0,T\mathclose{]}$ such that $\|x_{i}(\tilde{t}_{i})\|<d$. Indeed, let $\nu:=e_{i}$. From \eqref{eq-nu}, we have
\begin{equation*}
\lambda \int_{0}^{T} h_{i}(t, x(t))\,\mathrm{d}t + (1-\lambda) \int_{0}^{T} x_{i}(t)\, \mathrm{d}t =0
\end{equation*}
and, proceeding by contradiction, we easily reach the claim. Then, arguing as above, we have that $\|x\|_{\infty}< \sqrt{n} (d + K_{1}) =:R^{*}$.

As a consequence, the open ball $G=B(0,R^{*})$ is (trivially) a bound set for system \eqref{eq-L-0-lambda}.

\smallskip
\noindent
\textit{Step~3. Conclusion.}
From
\begin{equation*}
\|(\phi(x'))'\|_{L^{1}}
\leq \|\mathrm{Hess}\,\mathcal{G}(x)\|_{\infty} \|x'\|_{L^{1}}+ T \|h_{\lambda}(\cdot,x)\|_{\infty}  + \|p\|_{L^{1}},
\end{equation*}
we have that $(\phi(x'(\cdot)))'$ is bounded in $L^{1}$ and we conclude as above.
\end{proof}

\begin{remark}\label{rem-3.2}
Theorem~\ref{th-Lienard} in the variant $(i)$ provides an extension of \cite[Theorem~3.3]{PeXu-07}, where the result has been proved for $\phi=\phi_{p}$ and $h(t,x)=h(x)$ a conservative vector field satisfying the more restrictive condition $\langle h(x),x \rangle \geq b \|x\|^{\alpha} - c$ for every $x\in\mathbb{R}^{n}$, with $b>0$, $c\geq 0$ and $\alpha>1$.

Conditions $(ii)$ and $(iii)$ are classical ones in this context for vector second-order systems of Li\'{e}nard-type (cf.~\cite{GHMaWa-09,Ma-72,OmZa-84,Za-83}). Observe that both conditions are satisfied in the one-dimensional case if we assume the sign-condition $h(t,x) x \geq 0$ for every $t$ and $x$ with $|x|>d>0$. Thus, Theorem~\ref{th-Lienard} extends to a large class of $\phi$-Laplacian differential systems the classical theorem of Reissig mentioned above.
\hfill$\lhd$
\end{remark}

\begin{remark}\label{rem-3.3}
Some comments on the hypotheses of Theorem~\ref{th-Lienard} are in order. First of all, we notice that, in dimension $n=1$, condition $(i)$ implies the validity of $(ii)$ and $(iii)$, which are equivalent each other. On the other hand, the function
\begin{equation}\label{def-q}
q(x) = x e^{-|x|}, \quad x\in\mathbb{R},
\end{equation}
satisfies $(ii)$ and $(iii)$, but not $(i)$. In this case, $q(x) x \geq 0$ for every $x\in\mathbb{R}$ and thus $(\textsc{h}_{\textsc{h}}^{+})$ follows.

We focus on the case $n\geq 2$. In order to show that $(i)$, $(ii)$, $(iii)$ are independent, we are going to present three examples for which exactly one of the hypotheses is valid. Every example is given for $n=2$, however it can be easily generalized to treat the $n$-dimensional case for a general $n\geq 2$. Moreover, in each example condition $(\textsc{h}_{\textsc{h}}^{+})$ can be straightforwardly checked.

\smallskip
\noindent
\textit{Example~1.} Let $\varepsilon \in\mathopen{]}0,\pi^{-1}\mathclose{[}$. Let us consider the function
\begin{equation*}
h(x_{1},x_{2}) = \bigl{(}x_{1} - 2 x_{2} + \varepsilon \arctan x_{1}, x_{2} + \varepsilon \arctan x_{2} \bigr{)}, \quad x=(x_{1},x_{2})\in\mathbb{R}^{2}.
\end{equation*}
We observe that
\begin{equation*}
\langle h(x),x \rangle = (x_{1}- x_{2})^{2} + x_{1} \varepsilon \arctan x_{1}  + x_{2} \varepsilon \arctan x_{2} \geq  \varepsilon (x_{1}\arctan x_{1}  + x_{2} 
\arctan x_{2}),
\end{equation*}
which tends to $+\infty$ as $\|x\|\to+\infty$. Therefore, $(i)$ holds.
Moreover, we have
\begin{equation*}
\langle h(x+y),x \rangle = (x_{1}- x_{2})^{2} + x_{1} \varepsilon \arctan(x_{1}+y_{1})  + x_{2} \varepsilon \arctan(x_{2}+y_{2}) + x_{1} y_{1} + y_{2}(x_{2}- 2 x_{1}).
\end{equation*}
Let $x=(s,s)$, $s>0$, and $y=(-1,0)$. Therefore, 
\begin{equation*}
\langle h(x+y),x \rangle = \varepsilon s ( \arctan(s-1) + \arctan s) - s \leq s(\varepsilon\pi-1) < 0.
\end{equation*}
Hence, $(ii)$ does not hold.
Let again $x=(s,s)$, $s>0$, then $h_{1}(x) x_{1}= - s^{2} + \varepsilon s \arctan s < 0$ for $s$ sufficiently large. Therefore, $(iii)$ does not hold.

\smallskip
\noindent
\textit{Example~2.} Let us consider the function
\begin{equation*}
h(x_{1},x_{2}) = \bigl{(}(x_{1}-x_{2}) q(\|x\|), x_{2} q(\|x\|) \bigr{)}, \quad x=(x_{1},x_{2})\in\mathbb{R}^{2},
\end{equation*}
where $q$ is defined as in \eqref{def-q}. We observe that
\begin{equation*}
\langle h(x),x \rangle =  ( x_{1}^{2} - x_{1}x_{2} + x_{2}^{2} ) q(\|x\|) \leq \dfrac{3}{2} \|x\|^{2} q(\|x\|),
\end{equation*}
which tends to $0$ as $\|x\|\to+\infty$. Therefore, $(i)$ does not hold.
Moreover,
\begin{align*}
\langle h(x+y),x \rangle 
&=  ( x_{1}^{2} - x_{1}x_{2} + x_{2}^{2} + x_{1} y_{1} - x_{1} y_{2} + x_{2} y_{2}) q(\|x+y\|)  
\\
&\geq \biggl{(} \dfrac{1}{2} \|x\|^{2} - 3 \|x\|\|y\| \biggr{)} q(\|x+y\|) \geq 0
\end{align*}
whenever $\|x\|\geq 6 \|y\|$ and $\|y\|$ is bounded. Therefore, $(ii)$ holds.

For $x=(s,2s)$, $s>0$, we have $h_{1}(x)x_{1} = -s^{2} q(|s|) < 0$. Hence, $(iii)$ does not hold.

\smallskip
\noindent
\textit{Example~3.} Let us consider the function
\begin{equation*}
h(x_{1},x_{2}) = \bigl{(} (x_{1}^{3}-3x_{1}) q(|x_{1}|), (x_{2}^{3}-3x_{2}) q(|x_{2}|) \bigr{)}, \quad x=(x_{1},x_{2})\in\mathbb{R}^{2}.
\end{equation*}
Let $x=(s,1)$, $s>0$. The quantity
\begin{equation*}
\langle h(x),x \rangle = x_{1}^{2}(x_{1}^{2}-3) q(|x_{1}|) + x_{2}^{2}(x_{2}^{2}-3) q(|x_{2}|) = s^{2}(s^{2}-3)q(|s|) - 2 q(1)
\end{equation*}
tends to $- 2 e^{-1} <0$ as $s\to+\infty$. Therefore, $(i)$ and $(ii)$ (with $y=0$) do not hold. For every $i\in\{1,2\}$, we notice that
\begin{equation*}
h_{i}(x) x_{i} = x_{i}^{2}(x_{i}^{2}-3) q(|x_{i}|) \geq 0,
\end{equation*}
for every $x\in\mathbb{R}^{2}$ with $|x_{i}|>\sqrt{3}$. Therefore, $(iii)$ holds.
\hfill$\lhd$
\end{remark}

\begin{remark}\label{rem-3.4}
We observe that one could add to the list of hypotheses $(i)$--$(iii)$ in Theorem~\ref{th-Lienard} another classical condition, that is the \textit{generalized Villari condition} (cf.~\cite[p.~381]{MaMa-98}), which, in our context, reads as follows:
\begin{itemize}[leftmargin=28pt,labelsep=10pt]
\item there exists $d>0$ such that $\int_{0}^{T}h_{i}(t,x(t)) \,\mathrm{d}t \neq 0$ for some $i=1,\ldots,n$, for every $x=(x_{1},\ldots,x_{n})\in\mathbb{R}$ with $|x_{j}|> d$ for some $j=1,\ldots,n$.
\end{itemize}
It is worth noticing that condition~$(iii)$ with the strict inequalities is a special case of the generalized Villari condition. On the other hand, in Example~2 and Example~3 of Remark~\ref{rem-3.3}, if we consider the maps $h(t,x)$ multiplied by a scalar function $\rho(\|x\|)$ which vanishes outside a large open ball, the examples continue to satisfy $(ii)$ and, respectively, $(iii)$, but the generalized Villari condition does not hold.
\hfill$\lhd$
\end{remark}

\appendix

\section{Remarks on the convexity of a sublevel set}\label{appendix-A}

In this appendix, we propose an alternative proof of Lemma~\ref{lem-convex} based on a classical result about convex sets, namely the Tietze--Nakajima theorem \cite{Na-28,Ti-28}, that we recall for reader's convenience (see also \cite{Kl-51,SaStVa-61} for more general versions of the result).
A set $D\subseteq\mathbb{R}^{n}$ is \textit{locally convex} if each point of $D$ has a neighborhood whose intersection with $D$ is convex (cf.~\cite[Section~17]{La-82}).
Then, the following result holds (cf.~\cite{La-82,Va-64} for the proof).

\begin{theorem}[Tietze--Nakajima]
A closed connected locally convex set in a Euclidean space is convex.
\end{theorem}

We now give the proof of Lemma~\ref{lem-convex}.

\begin{proof}[Proof of Lemma~\ref{lem-convex}]
Let $V \colon \mathbb{R}^{n}\to \mathbb{R}$ be a $\mathcal{C}^{2}$-function satisfying \eqref{cond-V'} for some $c\in\mathbb{R}$ with $D=[V\leq c]$ nonempty, bounded and connected.
Let us assume that
\begin{itemize}[leftmargin=28pt,labelsep=10pt]
\item[$(C)$] for all $u\in\partial D$, $\langle V''(u)y,y\rangle \geq 0$,
for all $y\in \mathbb{R}^{n}$ with $\langle V'(u),y\rangle = 0$.
\end{itemize}
holds and we will prove that $D$ is convex.
By Tietze--Nakajima theorem it is sufficient to verify that $D$ is locally convex. Accordingly, let $u_{0}$ be an arbitrary but fixed point in $\partial D$ and we aim to prove that there exists a neighborhood $U$ of $u_{0}$ such that $D\cap U$ is convex. 
Recalling that $V'(u_{0})\neq 0$ by assumption \eqref{cond-V'}, we define the vector $w:=V'(u_{0})/\|V'(u_{0})\|$. Let $W:=\{z\in\mathbb{R}^{n} \colon \langle z,w\rangle=0\}$, that is the subspace (of dimension $n-1$) of $\mathbb{R}^{n}$ orthogonal to the vector $w$.
Since every $x\in\mathbb{R}^{n}$ can be uniquely written as $x=u_{0}+\alpha w + z$ for $(\alpha,z)\in\mathbb{R}\times W$, we can equivalently consider the function
\begin{equation*}
\mathcal{V}(\alpha,z):=V(u_{0}+\alpha w + z).
\end{equation*}
Thus, our goal reduces to prove that there exists a neighborhood $\mathcal{U}$ of $(0,0)$ in the $(\alpha,z)$-space $\mathbb{R}\times W$ 
such that $\mathcal{D}\cap \mathcal{U}$ is convex, where
\begin{equation*}
\mathcal{D}:=\bigl{\{} (\alpha,z)\in\mathbb{R}\times W\colon V(u_{0}+\alpha w+z)\leq c \bigr{\}}.
\end{equation*}
We observe that $\partial_{\alpha} \mathcal{V}(0,0)= \langle V'(u_{0}),w\rangle = \|V'(u_{0})\| >0$ and thus, by the implicit function theorem, there exist a neighborhood of $(0,0)$ of the form $\mathopen{]}-\varepsilon,\varepsilon\mathclose{[}\times B(0,\varepsilon)$ and a continuously differentiable map $\varphi\colon B(0,\varepsilon)\to\mathopen{]}-\varepsilon,\varepsilon\mathclose{[}$ such that
\begin{equation}\label{eq-IFT}
V(u_{0}+\alpha w + z) = c \quad \text{if and only if} \quad \alpha=\varphi(z),
\end{equation}
for every $(\alpha,z)\in \mathopen{]}-\varepsilon,\varepsilon\mathclose{[}\times B(0,\varepsilon)$. Here, $B(0,\varepsilon)$ is an open neighborhood of the origin in $W \cong \mathbb{R}^{n-1}$. Without loss of generality, we assume also that $\langle V'(\xi),w \rangle >0$ for all $\xi\in\mathopen{]}-\varepsilon,\varepsilon\mathclose{[}\times B(0,\varepsilon)$.
Hence, it will be sufficient to verify that the set
\begin{equation*}
\mathcal{D}\cap \mathcal{U} =\bigl{\{} (\alpha,z)\in \mathopen{]}-\varepsilon,\varepsilon\mathclose{[}\times B(0,\varepsilon) \colon \alpha \leq \varphi(z) \bigr{\}}
\end{equation*}
is convex. Let $(\alpha_{1},z_{1}),(\alpha_{2},z_{2})\in \mathcal{D}\cap \mathcal{U}$ and set $v:=(z_{2}-z_{1}) /\|z_{2}-z_{1}\|$ if $z_{1}\neq z_{2}$ and $v$ arbitrary with $\|v\|=1$ otherwise.
Then, to check the (local) convexity we can restrict ourselves to the intersection of $\mathcal{D}\cap \mathcal{U}$ with the $2$-dimensional subspace $\Sigma$ of $\mathbb{R}^{n}$ spanned by $v$ and $w$.
We set
\begin{equation*}
\ell(\beta) := V(u_{0}+\varphi(\beta v) w + \beta v) = c, \quad \vartheta(\beta):=\varphi(\beta v),
\quad \text{for every $\beta\in \mathopen{]}-\varepsilon,\varepsilon\mathclose{[}$.}
\end{equation*}
Then, for every $\beta\in \mathopen{]}-\varepsilon,\varepsilon\mathclose{[}$, we have
\begin{equation*}
\ell'(\beta) = \langle V'(u_{0}+\vartheta(\beta) w + \beta v),\vartheta'(\beta) w + v\rangle = \langle V'(u),y \rangle = 0,
\end{equation*}
where we have set
\begin{equation*}
u:=u_{0}+\vartheta(\beta) w + \beta v, \quad y:= \vartheta'(\beta) w + v.
\end{equation*}
Moreover, for every $\beta\in \mathopen{]}-\varepsilon,\varepsilon\mathclose{[}$, it holds that
\begin{align*}
\ell''(\beta) &=\langle V''(u_{0}+\vartheta(\beta) w + \beta v) (\vartheta'(\beta) w + v) , \vartheta'(\beta) w + v\rangle
\\
&\quad + \langle V'(u_{0}+\vartheta(\beta) w + \beta v), \vartheta''(\beta)w \rangle
\\
&= \langle V''(u)y,y \rangle + \vartheta''(\beta) \langle V'(u), w \rangle = 0.
\end{align*}
By hypothesis $(C)$ we conclude that $\vartheta''(\beta) \leq 0$ for all $\beta\in \mathopen{]}-\varepsilon,\varepsilon\mathclose{[}$, which implies that the subgraph of $\varphi$ restricted to $\mathcal{D}\cap\mathcal{U}\cap \Sigma$ is convex.

Conversely, assume that $D$ is convex. If condition $(C)$ is not satisfied there exists $u_{0}\in\partial D$ such that
$\langle V''(u_{0})v_{1},v_{1} \rangle <0$ for some vector $v_{1}\in\mathbb{R}^{n}$ with $\|v_{1}\|=1$ and $\langle V'(u_{0}),v_{1} \rangle=0$.
Arguing as in the previous part of the proof, we introduce the vector $w:=V'(u_{0})/\|V'(u_{0})\|$ and, using the implicit function theorem we find a neighborhood of $(0,0)$ of the form $\mathopen{]}-\varepsilon,\varepsilon\mathclose{[}\times B(0,\varepsilon)$ and a continuously differentiable map $\varphi\colon B(0,\varepsilon)\to\mathopen{]}-\varepsilon,\varepsilon\mathclose{[}$ such that \eqref{eq-IFT} holds for every $(\alpha,z)\in \mathopen{]}-\varepsilon,\varepsilon\mathclose{[}\times B(0,\varepsilon)$. As a next step, we consider the intersection of $\mathcal{D}\cap \mathcal{U}$ with the $2$-dimensional subspace $\Sigma$ of $\mathbb{R}^{n}$ spanned by $v_{1}$ and $w$.
By the convexity of $D$ it follows that the set
\begin{equation*}
\mathcal{A} := \bigl{\{}(\alpha,\beta)\in \mathopen{]}-\varepsilon,\varepsilon\mathclose{[}\times \mathopen{]}-\varepsilon,\varepsilon\mathclose{[} \colon \alpha \leq \varphi(\beta v_{1})\bigr{\}}
\end{equation*}
is convex. On the other hand, setting as above $\ell(\beta) := V(u_{0}+\varphi(\beta v_{1}) w + \beta v_{1})$ and $\vartheta(\beta):=\varphi(\beta v_{1})$, we obtain
\begin{align*}
0 &=\ell'(\beta) = \langle V'(u_{0}+\vartheta(\beta) w + \beta v_{1}), \vartheta'(\beta) w + v_{1}\rangle
\\
0 &=\ell''(\beta) =\langle V''(u_{0}+\vartheta(\beta) w + \beta v_{1}) (\vartheta'(\beta) w + v_{1}) ,  \vartheta'(\beta) w + v_{1}\rangle
\\
&\qquad \qquad \quad+ \langle V'(u_{0}+\vartheta(\beta) w + \beta v_{1}), \vartheta''(\beta)w \rangle.
\end{align*}
Hence, for $\beta=0$, we have that $\vartheta'(0)\|V'(u_{0})\|= - \langle V'(u_{0}),v_{1}\rangle =0$,
which implies $\vartheta'(0)=0$, and that
$\vartheta''(0) \|V'(u_{0})\| = - \langle V''(u_{0}) v_{1} ,v_{1} \rangle > 0$, which implies $\vartheta''(0)>0$.
Therefore, the map $\beta\mapsto\varphi(\beta v_{1})$ is strictly convex in a neighborhood $\mathopen{]}-\delta,\delta\mathclose{[}$ of the origin for a suitable $\delta\in\mathopen{]}0,\varepsilon\mathclose{]}$. This clearly contradicts the convexity of the set $\mathcal{A}$. The proof is completed.
\end{proof}

\bibliographystyle{elsart-num-sort}
\bibliography{FeZa-biblio}

\end{document}